\documentclass[letterpaper, 10pt, conference]{IEEEtran}  % Comment this line out
\IEEEoverridecommandlockouts                              % This command is only
\usepackage{epsfig}
\usepackage{amsmath}
\usepackage{amssymb}
\usepackage{indentfirst}
\usepackage{cite}
\usepackage{bm}
\usepackage{graphicx}           % for including graphics
\usepackage{psfrag}             % for changing text in eps graphics
\usepackage{subfigure}          % for subfigures
\usepackage{empheq}
\usepackage{enumerate}

\usepackage{amsthm}

\usepackage{tikz}
\usetikzlibrary{arrows}

\newenvironment{proof}[1][Proof]%
  {\smallskip\par\noindent\textbf{#1\,:\ }}%
  {\hspace*{\fill} \rule{6pt}{6pt}\smallskip}
\newenvironment{proof*}[1][Proof]%
  {\medskip\par\noindent\textbf{#1\,:\ }}%
\newtheorem{remark}{\textbf{Remark}}

\newtheorem{assumption}{Assumption}
\newtheorem{theorem}{Theorem}
\newtheorem{lemma}{Lemma}

\newtheorem{proposition}{Proposition}

\usepackage{color}
\definecolor{gray}{RGB}{128,128,128}

\usepackage{hyperref}

\title{Distributed Stochastic Gradient Method for Non-Convex Problems with Applications in Supervised Learning}

\author{J. George, T. Yang, H. Bai and P. Gurram% <-this % stops a space
\thanks{J. George is with U.S. Army Research Laboratory, Adelphi, MD 20783, USA.
{\tt\small jemin.george.civ@mail.mil}}%
\thanks{T.~Yang is with University of North Texas, Denton, TX 76203 USA.
{\tt\small Tao.Yang@unt.edu}}%
\thanks{He~Bai is with Oklahoma State University, Stillwater, OK 74078, USA.
{\tt\small he.bai@okstate.edu}}%
\thanks{P. Gurram is with Booz Allen Hamilton \& U.S. Army Research Laboratory, Adelphi, MD 20783, USA.
{\tt\small Gurram\_Prudhvi@bah.com}}%
}

\allowdisplaybreaks[4]

\begin{document}

\maketitle
\thispagestyle{empty}
\pagestyle{empty}

%%%%%%%%%%%%%%%%%%%%%%%%%%%%%%%%%%%%%%%%%%%%%%%%%%%%%%%%%%%%%%%%%%%%%%%%%%%%%%%%
\begin{abstract}
We develop a distributed stochastic gradient descent algorithm for solving non-convex optimization problems under the assumption that the local objective functions are twice continuously differentiable with Lipschitz continuous gradients and Hessians. We provide sufficient conditions on step-sizes that guarantee the asymptotic mean-square convergence of the proposed algorithm. We apply the developed algorithm to a distributed supervised-learning problem, in which a set of networked agents collaboratively train their individual neural nets to recognize handwritten digits in images. Results indicate that all agents report similar performance that is also comparable to the performance of a centrally trained neural net. Numerical results also show that the proposed distributed algorithm allows the individual agents to recognize the digits even though the training data corresponding to all the digits is not locally available to each agent.
\end{abstract}
%%%%%%%%%%%%%%%%%%%%%%%%%%%%%%%%%%%%%%%%%%%%%%%%%%%%%%%%%%%%%%%%%%%%%%%%%%%%%%%%

\section{Introduction}\label{sec:Intro}

With the advent of smart devices, there has been an exponential growth in the amount of data collected and stored locally on the individual devices. Applying machine learning to extract value from such massive data to provide data-driven insights, decisions, and predictions has been a hot research topic as well as the focus of numerous businesses like Google, Facebook, Alibaba, Yahoo, etc. However, porting these vast amounts of data to a data center to conduct traditional machine learning has raised two main issues: (i) the communication challenge associated with transferring vast amounts of data from a large number of devices to a central location and (ii) the privacy issues associated with sharing raw data. Distributed machine learning techniques based on the server-client architecture~\cite{Mu2014,ICCCN2017} have been proposed as solutions to this problem. On one extreme end of this architecture, we have the \emph{parameter server} approach, where a server or group of servers initiate distributed learning by pushing the current model to a set of client nodes that host the data. Client nodes compute the local gradients or parameter updates and communicate it to the server nodes. Server nodes aggregate these values and update the current model~\cite{Zhang2018, NIPS2014_5597}. On the other extreme, we have \emph{federated learning}, where each client node obtains a local solution to the learning problem and the server node computes a global model by simply averaging the local models~\cite{Jakub2016, McMahan2017}. These distributed learning techniques are not truly distributed since they follow a master-slave architecture and do not involve any peer-to-peer communication. Though these techniques are not always robust and they are rendered useless if the server fails, they do provide a good business opportunity for companies that own servers and host web services. However, our aim is to develop a fully distributed machine learning architecture enabled by client-to-client interaction.

For large-scale machine learning, stochastic gradient descent (SGD) methods are often preferred over batch gradient methods~\cite{Bottou2018SIAM} because (i) in many large-scale problems, there is a good deal of redundancy in data and therefore it is inefficient to use all the data in every optimization iteration, (ii) the computational cost involved in computing the batch gradient is much higher than that of the stochastic gradient, and (iii) stochastic methods are more suitable for online learning where data are arriving sequentially. Since most machine learning problems are non-convex, there is a need for distributed stochastic gradient methods for non-convex problems. Therefore, here we present a distributed stochastic gradient algorithm for non-convex problems and demonstrate its utility for distributed machine learning.

A few early examples of (non-stochastic or deterministic) distributed non-convex optimization algorithms include the Distributed Approximate Dual Subgradient (DADS) Algorithm \cite{Zhu2013}, NonconvEx primal-dual SpliTTing (NESTT) algorithm \cite{NESTT2016}, and the Proximal Primal-Dual Algorithm (Prox-PDA) \cite{pmlr-v70-hong17a}. More recently, a non-convex version of the accelerated distributed augmented Lagrangians (ADAL) algorithm is presented in \cite{Chatzipanagiotis2017} and successive convex approximation (SCA)-based algorithms such as iNner cOnVex Approximation (NOVA) and in-Network succEssive conveX approximaTion algorithm (NEXT) are given in \cite{Scutari2017} and \cite{Lorenzo2016}, respectively. References \cite{Hong2018, Guo2017, Hong2016SIAM} provide several distributed alternating direction method of multipliers (ADMM) based non-convex optimization algorithms. Non-convex versions of Decentralized Gradient Descent (DGD) and Proximal Decentralized Gradient Descent (Prox-DGD) are given in \cite{Zeng2018TSP}. Finally, Zeroth-Order NonconvEx (ZONE) optimization algorithms for mesh network (ZONE-M) and star network (ZONE-S) are presented in \cite{Hajinezhad2019}.

There exist several works on distributed stochastic gradient methods, but mainly for strongly convex optimization problems. These include the stochastic subgradient-push method for distributed optimization over time-varying directed graphs  given in \cite{Nedic2016}, distributed stochastic optimization over random networks given in \cite{2018arXiv180307836J}, the Stochastic Unbiased Curvature-aided Gradient (SUCAG) method given in \cite{SUCAG}, and distributed stochastic gradient tracking methods \cite{Pu2018ArxiV}. There are very few works on distributed stochastic gradient methods for non-convex optimization~\cite{Tatarenko2017,Bianchi2013}; however, they make very restrictive assumptions on the critical points of the problem.

Contributions of this paper are three-fold:
\begin{enumerate}
  \item We propose a fully distributed machine learning architecture that does not require any server nodes.
  \item We develop a distributed SGD algorithm and provide sufficient conditions on step-sizes such that the algorithm is mean-square convergent.
  \item We demonstrate the utility of the proposed SGD algorithm for distributed machine learning.
\end{enumerate}

\subsection{Notation}

Let $\mathbb{R}^{n\times m}$ denote the set of $n\times m$ real matrices. For a vector $\bm{\phi}$, $\phi_i$ is the $i^{\text{th}}$ entry of $\bm{\phi}$.  An $n\times n$ identity matrix is denoted as $I_n$ and $\mathbf{1}_n$ denotes an $n$-dimensional vector of all ones. For $p\in[1,\,\infty]$, the $p$-norm of a vector $\mathbf{x}$ is denoted as $\left\| \mathbf{x} \right\|_p$. For matrices $A \in \mathbb{R}^{m\times n}$ and $B \in \mathbb{R}^{p \times q}$, $A \otimes B \in \mathbb{R}^{mp \times nq}$ denotes their Kronecker product.

For a graph $\mathcal{G}\left(\mathcal{V},\mathcal{E}\right)$ of order $n$, $\mathcal{V} \triangleq \left\{v_1, \ldots, v_n\right\}$ represents the agents or nodes and the communication links between the agents are represented as $\mathcal{E} \triangleq \left\{e_1, \ldots, e_{\ell}\right\} \subseteq \mathcal{V} \times \mathcal{V}$. Let $\mathcal{A} = \left[a_{ij}\right]\in \mathbb{R}^{n\times n}$ be the \emph{adjacency matrix} with entries of $a_{ij} = 1 $ if $(v_i,v_j)\in\mathcal{E}$ and zero otherwise. Define $\Delta = \text{diag}\left(\mathcal{A}\mathbf{1}_n\right)$ as the in-degree matrix and $\mathcal{L} = \Delta - \mathcal{A}$ as the graph \emph{Laplacian}. 
\section{Distributed Machine Learning}\label{sec:Problem}

Our problem formulation closely follows the centralized machine learning problem discussed in \cite{Bottou2018SIAM}. Consider a networked set of $n$ agents, each with a set of $m_i$, $i=1,\ldots,n$, independently drawn input-output samples $\{\bm{x}_i^{j},\,\bm{y}_i^j\}_{j=1}^{j=m_i}$, where $\bm{x}_i^j\in\mathbb{R}^{d_x}$ and $\bm{y}_i^j\in\mathbb{R}^{d_y}$ are the $j$-th input and output data, respectively, associated with the $i$-th agent. For example, the input data could be images and the outputs could be labels. Let $h\left(\cdot\,;\,\cdot\right):\mathbb{R}^{d_x}\times\mathbb{R}^{d_w} \mapsto \mathbb{R}^{d_y}$, denote the prediction function, fully parameterized by the vector $\bm{w}\in\mathbb{R}^{d_w}$. Each agent aims to find the parameter vector that minimizes the losses, $\ell\left(\cdot\,;\,\cdot\right):\mathbb{R}^{d_y}\times\mathbb{R}^{d_y} \mapsto \mathbb{R}$, incurred from inaccurate predictions. Thus, the loss function $\ell\left(h\left(\bm{x}_i;\bm{w}\right),\bm{y}_i\right)$ yields the loss incurred by the $i$-th agent, where $h\left(\bm{x}_i;\bm{w}\right)$ and $\bm{y}_i$ are the predicted and true outputs, respectively for the $i$-th node.

%\begin{remark}
%  Note that the functional form of the prediction function $h\left(\cdot\,;\,\cdot\right)$ and the loss function $\ell\left(\cdot\,;\,\cdot\right)$ are common across all agents since the learning mechanism is identical across all agents, i.e., all agents have the same neural architecture.
%\end{remark}

Assuming the input output space $\mathbb{R}^{d_x}\times\mathbb{R}^{d_y}$ associated with the $i$-th agent is endowed with a probability measure $P_i~:~\mathbb{R}^{d_x}\times\mathbb{R}^{d_y}\mapsto[0,\,1]$, the objective function an agent wishes to minimize is
\begin{align}\label{Eq:Risk}
\begin{split}
R_i(\bm{w}) &= \int_{\mathbb{R}^{d_x}\times\mathbb{R}^{d_y}}\, \ell\left(h\left(\bm{x}_i;\bm{w}\right),\bm{y}_i\right)\,dP_i\left(\bm{x}_i,\bm{y}_i\right) \\
&= \mathbb{E}_{P_i} \left[ \ell\left(h\left(\bm{x}_i;\bm{w}\right),\bm{y}_i\right) \right].
\end{split}
\end{align}
Here $R_i(\bm{w})$ denotes the expected risk given a parameter vector $\bm{w}$ with respect to the probability distribution $P_i$. The total expected risk across all networked agents is given as
\begin{equation}\label{Eq:TotalRisk}
R(\bm{w}) = \sum_{i=1}^{n}\,R_i(\bm{w}) = \sum_{i=1}^{n}\, \mathbb{E}_{P_i} \left[ \ell\left(h\left(\bm{x}_i;\bm{w}\right),\bm{y}_i\right) \right].
\end{equation}
Minimizing the expected risk is desirable but often unattainable since the distributions $P_i$ are unknown. Thus, in practice each agent chooses to minimize the empirical risk $\bar{R}_i(\bm{w})$ defined as
\begin{equation}\label{Eq:Risk1}
\bar{R}_i(\bm{w}) = \frac{1}{m_i}\sum_{j=1}^{m_i} \ell\left(h\left(\bm{x}_i^j;\bm{w}\right),\bm{y}_i^j\right).
\end{equation}
Here, the assumption is that $m_i$ is large enough so that $\bar{R}_i(\bm{w}) \approx R_i(\bm{w})$. The total empirical risk across all networked agents is 
\begin{equation}\label{Eq:TotalRisk1}
\bar{R}(\bm{w}) = \sum_{i=1}^{n}\,\bar{R}_i(\bm{w}) = \sum_{i=1}^{n}\, \left[ \frac{1}{m_i}\sum_{j=1}^{m_i} \ell\left(h\left(\bm{x}_i^j;\bm{w}\right),\bm{y}_i^j\right) \right].
\end{equation}

In order to simplify the notation, let us represent a sample input-output pair $(\bm{x}_i,\,\bm{y}_i)$ by a random seed $\bm{\xi}_i$ and let $\bm{\xi}_i^j$ denotes the $j$-th sample associated with the $i$-th agent. Define the loss incurred for a given $\left(\bm{w},\bm{\xi}_i^j\right)$ as $\ell\left(\bm{w},\bm{\xi}_i^j\right)$. Now, the distributed learning problem can be posed as an optimization involving sum of local empirical risks, i.e.,
\begin{align}\label{eq:Opt1}
  \min_{\bm{w}}\, f(\bm{w}) = \min_{\bm{w}}\,\sum_{i=1}^{n} \, f_i\left( \bm{w} \right),
\end{align}
where $f_i\left( \bm{w} \right) = \frac{1}{m_i}\sum_{j=1}^{m_i} \ell\left(\bm{w},\bm{\xi}_i^{j}\right)$.
%\begin{equation}\label{Eq:Risk2}
%f_i\left( \bm{w} \right) = \frac{1}{m_i}\sum_{j=1}^{m_i} \ell\left(\bm{w},\bm{\xi}_i^{j}\right).
%\end{equation}

%@@@@@@@@@@@@@@@@@@@@@@@@@@@@@@@@@@@@@@@@@@@@@@@@@@@@@@@@@@@@@@@@@@@@@@@@@@@@@@@@@@@@@@@@@@@@@@@@@@@@@@@@@@
\section{Distributed SGD}\label{sec:SGD}
%@@@@@@@@@@@@@@@@@@@@@@@@@@@@@@@@@@@@@@@@@@@@@@@@@@@@@@@@@@@@@@@@@@@@@@@@@@@@@@@@@@@@@@@@@@@@@@@@@@@@@@@@@@

Here we propose a distributed stochastic gradient method to solve \eqref{eq:Opt1}. Let $\bm{w}_i(k) \in \mathbb{R}^{d_w}$ denote agent $i$'s estimate of the optimizer at time instant $k$. Thus, for an arbitrary initial condition $\bm{w}_i(0)$, the update rule at node $i$ is as follows:
\begin{align}\label{Eq:DSG}
\begin{split}
  \bm{w}_i(k + 1) = \bm{w}_i(k) &- \beta_k \, \sum_{j=1}^{n}\,a_{ij} \left( \bm{w}_i(k) - \bm{w}_j(k) \right) \\
  &- \alpha_k \,\mathbf{ g}_i\left( \bm{w}_i(k), \bm{\xi}_i(k) \right),
\end{split}
\end{align}
where $\alpha_k$ and $\beta_k$ are hyper parameters to be specified, $a_{ij}$ are the entries of the adjacency matrix and $\mathbf{g}_i\left( \bm{w}_i(k), \bm{\xi}_i(k) \right)$ represents either a simple stochastic gradient, mini-batch stochastic gradient or a stochastic quasi-Newton direction, i.e.,
\begin{align}\label{Eq:g}
  \mathbf{g}_i\left( \bm{w}_i(k), \bm{\xi}_i(k) \right) = \left\{
                                              \begin{array}{l}
                                                \nabla \ell\left(\bm{w}_i(k),\bm{\xi}_i^{k}\right), ~~\mbox{or} \\
                                                \frac{1}{n_i(k)}\sum\limits_{s=1}^{n_i(k)}\,\nabla \ell\left(\bm{w}_i(k),\bm{\xi}_i^{k,s}\right), ~~\mbox{or} \\
                                                H_i(k) \frac{1}{n_i(k)}\sum\limits_{s=1}^{n_i(k)}\,\nabla \ell\left(\bm{w}_i(k),\bm{\xi}_i^{k,s}\right),
                                              \end{array}
                                            \right.
\end{align}
\noindent where $n_i(k)$ denotes the mini-batch size, $H_i(k)$ is a positive definite scaling matrix, $\bm{\xi}_i^{k}$ represents the single random input-output pair sampled at time instant $k$, and $(\bm{\xi}_i^{k,s})$ denotes the $s$-th input-output pair out of the $n_i(k)$ random input-output pairs sampled at time instant $k$.

Define $\mathbf{w}(k) \triangleq \begin{bmatrix} \bm{w}_1^\top(k) & \ldots & \bm{w}_n^\top(k) \end{bmatrix}^\top \in \mathbb{R}^{nd_w}$. Now \eqref{Eq:DSG} can be written as
\begin{align}\label{Eq:DSG1}
  \mathbf{w}(k + 1) &= \left( \mathcal{W}_k \otimes I_{d_w}\right) \,\mathbf{w}(k) - \alpha_k \, \mathbf{g}(\mathbf{w}(k),\bm{\xi}(k)),
\end{align}
where $\mathcal{W}_k = \left(I_{n}-\beta_k\mathcal{L}\right)$, $\mathcal{L}$ is the network Laplacian and
$$\mathbf{g}(\mathbf{w}(k),\bm{\xi}(k)) \triangleq \begin{bmatrix}
                             \mathbf{g}_1\left( \bm{w}_1(k), \bm{\xi}_1(k) \right) \\
                             \vdots \\
                             \mathbf{g}_n\left( \bm{w}_n(k), \bm{\xi}_n(k) \right)
                           \end{bmatrix} \in \mathbb{R}^{n d_w}.$$

%@@@@@@@@@@@@@@@@@@@@@@@@@@@@@@@@@@@@@@@@@@@@@@@@@@@@@@@@@@@@@@@@@@@@@@@@@@@@@@@@@@@@@@@@@@@@@@@@@@@@@@@@@@
\subsection{Assumptions}
%@@@@@@@@@@@@@@@@@@@@@@@@@@@@@@@@@@@@@@@@@@@@@@@@@@@@@@@@@@@@@@@@@@@@@@@@@@@@@@@@@@@@@@@@@@@@@@@@@@@@@@@@@@

First, we state the following assumption on the individual objective functions:
%-------------------------------------------------------------------------------------------------------------------------------
\begin{assumption}\label{Assump:Lipz}
  Objective functions $f_i(\,\cdot\,)$ and its gradients $\nabla f_i(\,\cdot\,)$ $:\mathbb{R}^{d_w}\mapsto \mathbb{R}^{d_w}$ are Lipschitz continuous with Lipschitz constants $L^0_i > 0$ and $L_i > 0$, respectively, i.e., $\forall \,\bm{w}_a,\,\bm{w}_b\in\mathbb{R}^{d_w}, \, i=1,\ldots,n$, we have
\begin{align}\label{Eq:Lipz}
  \| f_i(\,\bm{w}_a\,) -  f_i(\,\bm{w}_b\,) \|_2 &\leq L_i^0 \|\bm{w}_a-\bm{w}_b\|_2\\
  \| \nabla f_i(\,\bm{w}_a\,) - \nabla f_i(\,\bm{w}_b\,) \|_2 &\leq L_i \|\bm{w}_a-\bm{w}_b\|_2.
\end{align}
\end{assumption}
%-------------------------------------------------------------------------------------------------------------------------------
%%-------------------------------------------------------------------------------------------------------------------------------
%\begin{assumption}\label{Assump:Lipz}
%  Objective functions $f_i(\,\cdot\,)$ are Lipschitz continuous with Lipschitz constants $L^0_i > 0$. Furthermore, the gradients $\nabla f_i(\,\cdot\,)$ $:\mathbb{R}^{d_w}\mapsto \mathbb{R}^{d_w}$ and the Hessians $\nabla^2 f_i(\,\cdot\,)$ $:\mathbb{R}^{d_w}\mapsto \mathbb{R}^{d_w\times d_w}$ are Lipschitz continuous with Lipschitz constants $L_i > 0$ and $L_{H_i}$, respectively, i.e., $\forall \,\bm{w}_a,\,\bm{w}_b\in\mathbb{R}^{d_w}, \, i=1,\ldots,n$, we have
%\begin{align}\label{Eq:Lipz}
%  \| f_i(\,\bm{w}_a\,) -  f_i(\,\bm{w}_b\,) \|_2 &\leq L_i^0 \|\bm{w}_a-\bm{w}_b\|_2\\
%  \| \nabla f_i(\,\bm{w}_a\,) - \nabla f_i(\,\bm{w}_b\,) \|_2 &\leq L_i \|\bm{w}_a-\bm{w}_b\|_2\\
%  \| \nabla^2 f_i(\,\bm{w}_a\,) - \nabla^2 f_i(\,\bm{w}_b\,) \|_2 &\leq L_{H_i} \|\bm{w}_a-\bm{w}_b\|_2.
%\end{align}
%\end{assumption}
%%-------------------------------------------------------------------------------------------------------------------------------
Now we introduce $F(\cdot):\mathbb{R}^{nd_w}\mapsto \mathbb{R}$, an aggregate objective function of local variables
\begin{align}\label{eq:obj}
   F(\mathbf{w}(k)) = \sum_{i=1}^{n} \, f_i\left( \bm{w}_i(k) \right).
\end{align}
Following Assumption \ref{Assump:Lipz}, the function $F(\cdot)$ is Lipschitz continuous with Lipschitz continuous gradient $\nabla F(\cdot)$, i.e., $\forall \,\mathbf{w}_a,\,\mathbf{w}_b\in\mathbb{R}^{n d_w}$, we have
\begin{equation}\label{Eq:Lipz1}
  \| \nabla F(\,\mathbf{w}_a\,) - \nabla F(\,\mathbf{w}_b\,) \|_2 \leq L \|\mathbf{w}_a-\mathbf{w}_b\|_2,
\end{equation}
with constant $L = \max\limits_i\{L_i\}$ and $\nabla F(\,\mathbf{w}\,) \triangleq \begin{bmatrix} \nabla f_1(\,\bm{w}_1\,)^\top &\ldots & \nabla f_n(\,\bm{w}_n\,)^\top \end{bmatrix}^\top \in \mathbb{R}^{n d_w}$.
%Furthermore, the Hessian is Lipschitz continuous, i.e.,
%\begin{equation}\label{Eq:Lipz1a}
%  \| \nabla^2 F(\,\mathbf{w}_a\,) - \nabla^2 F(\,\mathbf{w}_b\,) \|_2 \leq L_H \|\mathbf{w}_a-\mathbf{w}_b\|_2, \quad \forall \,\mathbf{w}_a,\,\mathbf{w}_b\in\mathbb{R}^{n d_w},
%\end{equation}
%with constant $L_H = \max\limits_i\{L_{H_i}\}$.
%~~~~~~~~~~~~~~~~~~~~~~~~~~~~~~~~~~~~~~~~~~~~~~~~~~~~~~~~~~~~~~~~~~~~~~~~~~~~~~~~~~~~~~~~~~~~~~~~~~~~~~~~~~~~~~~~~~~~~~~~~~~~~~~~~~
\begin{lemma}\label{Lemma:0}
  Given Assumption \ref{Assump:Lipz}, we have
\begin{align}\label{Eqn:lemma0}
  \| \nabla F\left(\mathbf{w}\right) \|_2 \leq \mu_F, \,\, \forall \mathbf{w}\in\mathbb{R}^{n d_w},
\end{align}
where $\mu_F < \infty$ is a positive constant.
\end{lemma}
%~~~~~~~~~~~~~~~~~~~~~~~~~~~~~~~~~~~~~~~~~~~~~~~~~~~~~~~~~~~~~~~~~~~~~~~~~~~~~~~~~~~~~~~~~~~~~~~~~~~~~~~~~~~~~~~~~~~~~~~~~~~~~~~~~~
\begin{proof}
  See Lemma 3.3 in~\cite{khalil2002nonlinear}.
\end{proof}
%~~~~~~~~~~~~~~~~~~~~~~~~~~~~~~~~~~~~~~~~~~~~~~~~~~~~~~~~~~~~~~~~~~~~~~~~~~~~~~~~~~~~~~~~~~~~~~~~~~~~~~~~~~~~~~~~~~~~~~~~~~~~~~~~~~
\begin{lemma}\label{Lemma:Lipz}
  Given Assumption \ref{Assump:Lipz}, we have $\forall \,\mathbf{w}_a,\,\mathbf{w}_b\in\mathbb{R}^{n d_w}$,
\begin{align}\label{Eqn:lemma1}
  F(\,\mathbf{w}_b\,)
&\leq F(\,\mathbf{w}_a\,) + \nabla F\left(\mathbf{w}_a\right)^\top(\mathbf{w}_b-\mathbf{w}_a) +
\frac{1}{2} L  \|\mathbf{w}_b-\mathbf{w}_a\|_2^2.
\end{align}
\end{lemma}
%~~~~~~~~~~~~~~~~~~~~~~~~~~~~~~~~~~~~~~~~~~~~~~~~~~~~~~~~~~~~~~~~~~~~~~~~~~~~~~~~~~~~~~~~~~~~~~~~~~~~~~~~~~~~~~~~~~~~~~~~~~~~~~~~~~
\begin{proof}
  Proof follows from the mean value theorem.
\end{proof}
%\begin{proof}
%Under Assumption \ref{Assump:Lipz}, one obtains
%\begin{align*}
% F(\,\mathbf{w}_b\,) &=F(\,\mathbf{w}_a\,) + \int_{0}^{1} \, \frac{\partial F\left(\mathbf{w}_a + t(\mathbf{w}_b-\mathbf{w}_a)\right)}{\partial t} \,dt \\
%&=F(\,\mathbf{w}_a\,) + \int_{0}^{1} \, \nabla F\left(\mathbf{w}_a + t(\mathbf{w}_b-\mathbf{w}_a)\right)^\top (\mathbf{w}_b-\mathbf{w}_a) \,dt, \\
%&=F(\,\mathbf{w}_a\,) + \nabla F\left(\mathbf{w}_a\right)^\top(\mathbf{w}_b-\mathbf{w}_a) + \int_{0}^{1} \, \left[ \nabla F\left(\mathbf{w}_a + t(\mathbf{w}_b-\mathbf{w}_a)\right) - \nabla F\left(\mathbf{w}_a\right)\right]^\top (\mathbf{w}_b-\mathbf{w}_a) \,dt, \\
%&\leq F(\,\mathbf{w}_a\,) + \nabla F\left(\mathbf{w}_a\right)^\top(\mathbf{w}_b-\mathbf{w}_a) +
%\int_{0}^{1} \, L \| t(\mathbf{w}_b-\mathbf{w}_a) \|_2 \|\mathbf{w}_b-\mathbf{w}_a\|_2 \,dt.
%\end{align*}
%Thus we have
%\begin{align*}
% F(\,\mathbf{w}_b\,)
%&\leq F(\,\mathbf{w}_a\,) + \nabla F\left(\mathbf{w}_a\right)^\top(\mathbf{w}_b-\mathbf{w}_a) +
%\frac{1}{2} L  \|\mathbf{w}_b-\mathbf{w}_a\|_2^2.
%\end{align*}
%\end{proof}
%-------------------------------------------------------------------------------------------------------------------------------
\begin{assumption}\label{Assump:Fmin}
  The function $F(\cdot)$ is lower bounded by $F_{\inf}$, i.e.,
\begin{equation}\label{Fbound}
  F_{\inf} \leq F(\mathbf{w}), \quad \forall\,\mathbf{w}\in\mathbb{R}^{nd_w}
\end{equation}
\end{assumption}
%-------------------------------------------------------------------------------------------------------------------------------
Without loss of generality, we assume that $F_{\inf} \geq 0$.
\noindent Now we make the following assumption regarding $\{\alpha_k\}$ and $\{\beta_k\}$:
%~~~~~~~~~~~~~~~~~~~~~~~~~~~~~~~~~~~~~~~~~~~~~~~~~~~~~~~~~~~~~~~~~~~~~~~~~~~~~~~~~~~~~~~~~~~~~~~~~~~~~~~~~~~~~~~~~~~~~~~~~~~~~~~~~~
\begin{assumption}\label{Assump:AlphaBeta}
Sequences $\{\alpha_k\}$ and $\{\beta_k\}$ are selected as
\begin{align}\label{Eqn:AlphaBeta}
\alpha_k = \frac{a}{(k+1)^{\delta_2}} \quad \textnormal{and} \quad \beta_k = \frac{b}{(k+1)^{\delta_1}},
\end{align}
where $a > 0$, $b>0$, $0 < 3\delta_1 < \delta_2 \leq 1$, $\delta_1 + \delta_2 > 1$, and $\delta_2 > 1/2$.
\end{assumption}
%~~~~~~~~~~~~~~~~~~~~~~~~~~~~~~~~~~~~~~~~~~~~~~~~~~~~~~~~~~~~~~~~~~~~~~~~~~~~~~~~~~~~~~~~~~~~~~~~~~~~~~~~~~~~~~~~~~~~~~~~~~~~~~~~~~
For sequences $\{\alpha_k\}$ and $\{\beta_k\}$ that satisfy Assumption \ref{Assump:AlphaBeta}, we have $\sum\limits_{k=1}^{\infty}\,\alpha_k = \infty$, $\sum\limits_{k=1}^{\infty}\,\beta_k = \infty$, $\sum\limits_{k=1}^{\infty}\,\alpha_k^2 < \infty$ and $\sum\limits_{k=1}^{\infty}\,\alpha_k\beta_k < \infty$. Thus $\alpha_k$ and $\beta_k$ are not summable sequences. However, $\alpha_k$ is square-summable and $\alpha_k\beta_k$ is summable.
%-------------------------------------------------------------------------------------------------------------------------------
\begin{assumption}\label{Assump:Graph}
  The interaction topology of $n$ networked agents is given as a connected undirected graph $\mathcal{G}\left(\mathcal{V},\mathcal{E}\right)$.
\end{assumption}
%-------------------------------------------------------------------------------------------------------------------------------

%For the connected undirected graph $\mathcal{G}\left(\mathcal{V},\mathcal{E}\right)$, $\mathcal{L}$ is a positive semi-definite matrix with one eigenvalue at 0 corresponding to the eigenvector $\mathbf{1}_n$. Furthermore, ~(Lemma~3~\cite{Gutman04}).
%
%\begin{remark}
%For all $\mathbf{x}\in\mathbb{R}^n$, such that $\mathbf{1}^T_n\mathbf{x} = 0$, we have $\mathbf{x}^T\mathcal{L}\left( \mathcal{L} \right)^+\mathbf{x} =  \mathbf{x}^T\mathbf{x} $.
%\end{remark}

%~~~~~~~~~~~~~~~~~~~~~~~~~~~~~~~~~~~~~~~~~~~~~~~~~~~~~~~~~~~~~~~~~~~~~~~~~~~~~~~~~~~~~~~~~~~~~~~~~~~~~~~~~~~~~~~~~~~~~~~~~~~~~~~~~~
\begin{lemma}\label{Lemma2}
  Given Assumption \ref{Assump:Graph}, for all $\mathbf{x}\in\mathbb{R}^n$ we have
\begin{align}\label{Eqn:lemma2eq2}
  \mathbf{x}^\top\mathcal{L}\mathbf{x} = \tilde{\mathbf{x}}^\top\mathcal{L}\tilde{\mathbf{x}}  \geq  \lambda_2(\mathcal{L})\|\tilde{\mathbf{x}}\|_2^2,
\end{align}
where $\tilde{\mathbf{x}} = \left(I_n - \frac{1}{n}\mathbf{1}_n\mathbf{1}^\top_n\right) \mathbf{x}$ is the average-consensus error and $\lambda_2(\cdot)$ denotes the smallest non-zero eigenvalue.
\end{lemma}
%~~~~~~~~~~~~~~~~~~~~~~~~~~~~~~~~~~~~~~~~~~~~~~~~~~~~~~~~~~~~~~~~~~~~~~~~~~~~~~~~~~~~~~~~~~~~~~~~~~~~~~~~~~~~~~~~~~~~~~~~~~~~~~~~~~
\begin{proof}
  This Lemma follows from the Courant-Fischer Theorem~\cite{Horn}.
\end{proof}
%~~~~~~~~~~~~~~~~~~~~~~~~~~~~~~~~~~~~~~~~~~~~~~~~~~~~~~~~~~~~~~~~~~~~~~~~~~~~~~~~~~~~~~~~~~~~~~~~~~~~~~~~~~~~~~~~~~~~~~~~~~~~~~~~~~
\begin{assumption}\label{Assump:Beta}
Parameter $b$ in sequence $\{\beta_k\}$ is selected such that
\begin{align}
\mathcal{W}_0 = \left(I_{n}-b\mathcal{L}\right)
\end{align}
has a single eigenvalue at $1$ corresponding to the right eigenvector $\mathbf{1}_n$ and the remaining $n-1$ eigenvalues of $\mathcal{W}_0$ are strictly inside the unit circle.
\end{assumption}
%~~~~~~~~~~~~~~~~~~~~~~~~~~~~~~~~~~~~~~~~~~~~~~~~~~~~~~~~~~~~~~~~~~~~~~~~~~~~~~~~~~~~~~~~~~~~~~~~~~~~~~~~~~~~~~~~~~~~~~~~~~~~~~~~~~
In other words, $b$ is selected such that $b < 1/\sigma_{\max}(\mathcal{L})$, where $\sigma_{\max}(\cdot)$ denotes the largest singular value. Thus, $b\sigma_{\max}(\mathcal{L}) < 1$.

Let $\mathbb{E}_{\xi}[\cdot]$ denote the expected value taken with respect to the distribution of the random variable $\bm{\xi}_k$ given the filtration $\mathcal{F}_{k}$ generated by the sequence $\{\mathbf{w}_{0},\ldots,\mathbf{w}_{k}\}$, i.e.,
\begin{align*}
   \mathbb{E}_{\xi}[\,\mathbf{w}_{k+1}\,] &= \mathbb{E}[\,\mathbf{w}_{k+1}\,|\mathcal{F}_{k}] \\
   &= \left(\mathcal{W}_k \otimes I_{d_w} \right) \mathbf{w}_k - \alpha_k \mathbb{E}[\,\mathbf{g}(\mathbf{w}_k,\bm{\xi}_k)\,|\mathcal{F}_{k}]\,\,\textnormal{a.s.},
\end{align*}
where $\textnormal{a.s.}$ (almost surely) denote events that occur with probability one. Now we make the following assumptions regarding the stochastic gradient term $\mathbf{g}(\mathbf{w}(k),\bm{\xi}(k))$.
%~~~~~~~~~~~~~~~~~~~~~~~~~~~~~~~~~~~~~~~~~~~~~~~~~~~~~~~~~~~~~~~~~~~~~~~~~~~~~~~~~~~~~~~~~~~~~~~~~~~~~~~~~~~~~~~~~~~~~~~~~~~~~~~~~~
\begin{assumption}\label{Assump:Grad1}
  Stochastic gradients are unbiased such that
\begin{equation}
  \mathbb{E}_{\xi}\left[\,  \mathbf{g}(\mathbf{w}_k,\bm{\xi}_k) \,\right] = \nabla F(\mathbf{w}_k), \quad \textnormal{a.s.}
\end{equation}
That is to say
\begin{align*}
  \mathbb{E}_{\xi}\left[\,  \mathbf{g}(\mathbf{w}_k,\bm{\xi}_k) \,\right] &=
\begin{bmatrix}
\mathbb{E}_{\xi_1}\left[\,\mathbf{g}_1\left( \bm{w}_1(k), \bm{\xi}_1(k) \right) \,\right]\\
\vdots \\
\mathbb{E}_{\xi_n}\left[\,\mathbf{g}_n\left( \bm{w}_n(k), \bm{\xi}_n(k) \right) \,\right]
\end{bmatrix}
=
\begin{bmatrix}
  \nabla f_1(\,\bm{w}_1(k)\,) \\
  \vdots \\
  \nabla f_n(\,\bm{w}_n(k)\,)
\end{bmatrix}% = \nabla F(\mathbf{w}_k), \quad \textnormal{a.s.}
\end{align*}
\end{assumption}
%~~~~~~~~~~~~~~~~~~~~~~~~~~~~~~~~~~~~~~~~~~~~~~~~~~~~~~~~~~~~~~~~~~~~~~~~~~~~~~~~~~~~~~~~~~~~~~~~~~~~~~~~~~~~~~~~~~~~~~~~~~~~~~~~~~
%~~~~~~~~~~~~~~~~~~~~~~~~~~~~~~~~~~~~~~~~~~~~~~~~~~~~~~~~~~~~~~~~~~~~~~~~~~~~~~~~~~~~~~~~~~~~~~~~~~~~~~~~~~~~~~~~~~~~~~~~~~~~~~~~~~
\begin{assumption}\label{Assump:Grad2}
 Stochastic gradients have conditionally bounded second moment, i.e., there exist scalars $\bar{\mu}_{v_1} \geq 0$ and $\bar{\mu}_{v_2} \geq 0$ such that
\begin{align}
\begin{split}
\mathbb{E}_{\xi} \left[ \| \mathbf{g}(\mathbf{w}_k,\bm{\xi}_k) \|_2^2 \right] \leq
\bar{\mu}_{v_1} + \bar{\mu}_{v_2} \left\|  \nabla F(\mathbf{w}_k) \right\|^2_2,\quad \textnormal{a.s.}.
\end{split}\label{Eq:2ndGrad}
\end{align}
\end{assumption}
%~~~~~~~~~~~~~~~~~~~~~~~~~~~~~~~~~~~~~~~~~~~~~~~~~~~~~~~~~~~~~~~~~~~~~~~~~~~~~~~~~~~~~~~~~~~~~~~~~~~~~~~~~~~~~~~~~~~~~~~~~~~~~~~~~~
Assumption~\ref{Assump:Grad2} is the bounded variance assumption typically make in SGD literature. Finally, it follows from Assumptions~\ref{Assump:Lipz}, \ref{Assump:Grad2} and Lemma \ref{Lemma:0} that the stochastic gradients are bounded, which is usually just assumed in literature~\cite{Nedic2009TAC,Bottou2018SIAM,Tatarenko2017,Zeng2018TSP}.
%~~~~~~~~~~~~~~~~~~~~~~~~~~~~~~~~~~~~~~~~~~~~~~~~~~~~~~~~~~~~~~~~~~~~~~~~~~~~~~~~~~~~~~~~~~~~~~~~~~~~~~~~~~~~~~~~~~~~~~~~~~~~~~~~~~
\begin{proposition}\label{Assump:BoundedGrad}
There exists a positive constant $\mu_g < \infty$ such that
\begin{align}
\sup\limits_{k\geq0} \, \mathbb{E} \left[ \| \mathbf{g}(\mathbf{w}_k,\bm{\xi}_k) \|_2^2 \right] \leq \mu_g.
\end{align}
\end{proposition}
%~~~~~~~~~~~~~~~~~~~~~~~~~~~~~~~~~~~~~~~~~~~~~~~~~~~~~~~~~~~~~~~~~~~~~~~~~~~~~~~~~~~~~~~~~~~~~~~~~~~~~~~~~~~~~~~~~~~~~~~~~~~~~~~~~~
\begin{proof}
Proof follows from taking the expectation of \eqref{Eq:2ndGrad} and applying the result from Lemma~\ref{Lemma:0}.
\end{proof}

\section{Convergence Analysis}

Our strategy for proving the convergence of the proposed distributed SGD algorithm to a critical point is as follows. First we show that the consensus error among the agents are diminishing at the rate of $O\left(\frac{1}{(k+1)^{\delta_2}}\right)$ (see Theorem~\ref{Theorem:Consensus}). Asymptotic convergence of the algorithm is then proved in Theorem~\ref{Theorem:Convergence}. Theorem~\ref{Theorem:SummableGrad} then establishes that the weighted expected average gradient norm is a summable  sequence. Finally, Theorem~\ref{Theorem:OptCond} proves the asymptotic mean-square convergence of the algorithm to a critical point.

%~~~~~~~~~~~~~~~~~~~~~~~~~~~~~~~~~~~~~~~~~~~~~~~~~~~~~~~~~~~~~~~~~~~~~~~~~~~~~~~~~~~~~~~~~~~~~~~~~~~~~~~~~~~~~~~~~~~~~~~~~~~~~~~~~~
\begin{theorem}\label{Theorem:Consensus}
%  Given Assumptions \ref{Assump:Lipz}, \ref{Assump:Fmin}, \ref{Assump:AlphaBeta}, \ref{Assump:Graph}, \ref{Assump:Beta}, \ref{Assump:Grad1}, and \ref{Assump:Grad2}, the distributed SGD algorithm
%  \begin{align}\label{Eq:DSG1a}
%    \mathbf{w}(k + 1) &= \left( \mathcal{W}_k \otimes I_{d_w}\right) \,\mathbf{w}(k) - \alpha_k \, \mathbf{g}(\mathbf{w}(k),\bm{\xi}(k)),
%  \end{align}
%  guarantees that
  Consider distributed SGD algorithm \eqref{Eq:DSG1} under Assumptions [1-7]. Then, there holds:
  \begin{align}\label{Eq:MSbound1a}
      \mathbb{E} \left[ \| \tilde{\mathbf{w}}_k \|_2^2 \right] = O\left(\displaystyle\frac{1}{(k+1)^{\delta_2}}\right).
  \end{align}
\end{theorem}
%~~~~~~~~~~~~~~~~~~~~~~~~~~~~~~~~~~~~~~~~~~~~~~~~~~~~~~~~~~~~~~~~~~~~~~~~~~~~~~~~~~~~~~~~~~~~~~~~~~~~~~~~~~~~~~~~~~~~~~~~~~~~~~~~~~
\begin{proof}
See Appendix~\textsc{B}.
\end{proof}

\noindent Let
\begin{equation}
    \gamma_k = \frac{\alpha_k}{\beta_k} = \frac{a/b}{(k+1)^{\delta_2-\delta_1}}.
\end{equation}
Now define a non-negative function $V(\gamma_k,\mathbf{w}_k)$ as
\begin{equation}\label{Vk}
  V(\gamma_k,\mathbf{w}_k)
  = F(\mathbf{w}_k) + \frac{1}{2\gamma_k}\, \mathbf{w}^\top_k \left(\mathcal{L}\otimes I_{d_w}\right)\mathbf{w}_k.
\end{equation}
Now taking the gradient with respect to $\mathbf{w}$ yields
\begin{equation}\label{dVk}
  \nabla V(\gamma_k,\mathbf{w}_k) = \nabla F(\mathbf{w}_k) + \frac{1}{\gamma_k}\, \left(\mathcal{L}\otimes I_{d_w}\right)\mathbf{w}_k.
\end{equation}
%~~~~~~~~~~~~~~~~~~~~~~~~~~~~~~~~~~~~~~~~~~~~~~~~~~~~~~~~~~~~~~~~~~~~~~~~~~~~~~~~~~~~~~~~~~~~~~~~~~~~~~~~~~~~~~~~~~~~~~~~~~~~~~~~~~
\begin{theorem}\label{Theorem:InfSum}
  Consider distributed SGD algorithm \eqref{Eq:DSG1} under Assumptions [1-7]. Then, for the gradient $\nabla V(\gamma_k,\mathbf{w}_k)$ given in \eqref{dVk}, there holds:
  \begin{align}\label{Eqn:SummableGrad}
      \sum\limits_{k=0}^{\infty} \, \alpha_k \mathbb{E}\left[ \left\|  \nabla V(\gamma_k,\mathbf{w}_k)  \right\|^2_2 \right] < \infty.
  \end{align}
\end{theorem}
%~~~~~~~~~~~~~~~~~~~~~~~~~~~~~~~~~~~~~~~~~~~~~~~~~~~~~~~~~~~~~~~~~~~~~~~~~~~~~~~~~~~~~~~~~~~~~~~~~~~~~~~~~~~~~~~~~~~~~~~~~~~~~~~~~~
\begin{proof}
See Appendix~\textsc{C}.
\end{proof}

\begin{theorem}\label{Theorem:Convergence}
  For the distributed SGD algorithm \eqref{Eq:DSG1} under Assumptions [1-7] we have
      \begin{align}\label{InEq:DSG3}
      \sum\limits_{k=0}^{\infty} \, \mathbb{E} \left[ \left\| \mathbf{w}_{k + 1} - \mathbf{w}_{k} \right\|_2^2 \right] &< \infty.
    \end{align}
    and
    \begin{align}\label{ConvergenceEqn}
      \lim_{k\rightarrow\infty} \, \mathbb{E} \left[ \left\| {\mathbf{w}}_{k + 1} - {\mathbf{w}}_{k} \right\|_2^2 \right] =0.
    \end{align}
\end{theorem}
\begin{proof}
See Appendix~\textsc{D}.
\end{proof}

Define $\bar{\mathbf{w}}_{k} = \displaystyle \frac{1}{n} \left(\mathbf{1}_n \mathbf{1}_n^\top \otimes I_{d_w} \right) \mathbf{w}_{k}$ and $ \overline{\nabla F} (\mathbf{w}_k) = \displaystyle \frac{1}{n} \left(\mathbf{1}_n \mathbf{1}_n^\top \otimes I_{d_w} \right) \nabla F(\mathbf{w}_k)$.

\begin{theorem}\label{Theorem:SummableGrad}
  For the distributed SGD algorithm \eqref{Eq:DSG1} under Assumptions [1-7] we have
  \begin{equation}\label{Eqn:SummableDFbar}
    \sum\limits_{k=0}^{\infty} \, \alpha_k\, \mathbb{E}\left[ \left\| \overline{\nabla F} (\mathbf{w}_k) \right\|^2_2 \right] < \infty.
  \end{equation}
\end{theorem}

\begin{proof}
See Appendix~\textsc{E}.
\end{proof}

Theorem \ref{Theorem:SummableGrad} establishes results about the weighted sum of expected average gradient norm and the key takeaway from this result is that, for the distributed SGD in \eqref{Eq:DSG1} with appropriate step-sizes, the expected average gradient norms cannot stay bounded away from zero (See Theorem 9 of \cite{Bottou2018SIAM}), i.e.,
\begin{align}\label{LimInf}
  \liminf_{k\rightarrow\infty}\,\mathbb{E}\left[ \left\| \overline{\nabla F} (\mathbf{w}_k) \right\|^2_2 \right] = 0.
\end{align}
Finally, we present the following result to illustrate that stronger convergence results follows from the continuity assumption on the Hessian, which has not been utilized in our analysis so far.
%-------------------------------------------------------------------------------------------------------------------------------
\begin{assumption}\label{Assump:LipzHess}
  The Hessians $\nabla^2 f_i(\,\cdot\,)$ $:\mathbb{R}^{d_w}\mapsto \mathbb{R}^{d_w\times d_w}$ are Lipschitz continuous with Lipschitz constants $L_{H_i}$, i.e., $\forall \,\bm{w}_a,\,\bm{w}_b\in\mathbb{R}^{d_w}, \, i=1,\ldots,n$, we have
\begin{equation}\label{Eq:LipzHess}
  \| \nabla^2 f_i(\,\bm{w}_a\,) - \nabla^2 f_i(\,\bm{w}_b\,) \|_2 \leq L_{H_i} \|\bm{w}_a-\bm{w}_b\|_2.
\end{equation}
\end{assumption}
%-------------------------------------------------------------------------------------------------------------------------------
It follows from Assumption~\ref{Assump:LipzHess} that the Hessian $\nabla^2 F(\cdot)$ is Lipschitz continuous, i.e., $\forall \,\mathbf{w}_a,\,\mathbf{w}_b\in\mathbb{R}^{n d_w}$,
\begin{equation}\label{Eq:LipzHess1}
  \| \nabla^2 F(\,\mathbf{w}_a\,) - \nabla^2 F(\,\mathbf{w}_b\,) \|_2 \leq L_H \|\mathbf{w}_a-\mathbf{w}_b\|_2,
\end{equation}
with constant $L_H = \max\limits_i\{L_{H_i}\}$.

%-------------------------------------------------------------------------------------------------------------------------------
%-------------------------------------------------------------------------------------------------------------------------------
\begin{theorem}\label{Theorem:OptCond}
  For the distributed SGD algorithm \eqref{Eq:DSG1} under Assumptions [1-8] we have
    \begin{align}\label{InEq:Opt}
        \lim_{k\rightarrow\infty}\, \mathbb{E}\left[\, \left\| \overline{\nabla F} (\mathbf{w}_k) \right\|^2_2 \,\right] = 0.
    \end{align}
\end{theorem}
%-------------------------------------------------------------------------------------------------------------------------------
%-------------------------------------------------------------------------------------------------------------------------------
\begin{proof}
  See Appendix~\textsc{F}
\end{proof}

\begin{remark}
Similar to the centralized SGD~\cite{Bottou2018SIAM}, the analysis given here shows the mean-square convergence of the distributed algorithm to a critical point, which include the saddle points. Though SGD has shown to escape saddle points efficiently~\cite{Lee2017, 2019arXiv190200247F, 2019arXiv190204811J}, extension of such results for distributed SGD is currently nonexistent and is the topic of future research.
\end{remark} 
\section{Application to Distributed Supervised Learning}\label{simulation}

We apply the proposed algorithm for distributedly training 10 different neural nets to recognize handwritten digits in images. Specifically, we consider a subset of the MNIST\footnote{\tt\small{http://yann.lecun.com/exdb/mnist/}} data set containing 5000 images of 10 digits (0-9), of which 2500 are used for training and 2500 are used for testing. Training data are divided among ten agents connected in an undirected unweighted ring topology~(see Fig.~\ref{Fig:NNPic}).
%%%%%%%%%%%%%%%%%%%%%%%%%%%%%%%%%%%%%%%%%%%%%%%%%%%%%%%%%%%%%%%%%%%%%%%%%%%%
\begin{figure}[h]
  \begin{centering}
      {\includegraphics[width=.45\textwidth]{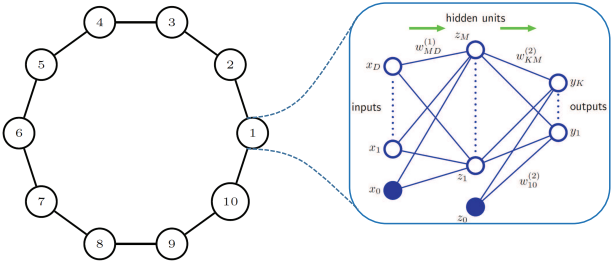}
      \caption{Network of 10 agents, each with its own neural net.}\label{Fig:NNPic}}
  \end{centering}

\end{figure}
%%%%%%%%%%%%%%%%%%%%%%%%%%%%%%%%%%%%%%%%%%%%%%%%%%%%%%%%%%%%%%%%%%%%%%%%%%%%
Each agent aims to train its own neural network consisting of a single hidden layer of 50 neurons (51 including the bias neuron). Since the images are $20\times20$, the input layer consists of 401 neurons (including the one bias neuron) and the output later consists of 10 neurons, one for each output class, i.e., one for each digits 0-9. As shown in Fig.~\ref{Fig:NNPic}, for each agent, the neural net consists of two sets of weights $W^{(1)} \in \mathbb{R}^{50\times401}$ and $W^{(2)} \in \mathbb{R}^{10\times51}$. Here $W^{(1)}$ links the input layer to the hidden layer and $W^{(2)}$ connects the hidden layer to the output later. We use a logistic sigmoid function for both the hidden unit activation and the output unit activation. Therefore, the input to output mapping for the neural net under consideration takes the form
\begin{align*}
    y_{\kappa}\left(\mathbf{x},\mathbf{w}\right) = h\left( \sum_{j=0}^{50} w^{(2)}_{{\kappa},j}\,h\left( \sum_{i=0}^{400} w^{(1)}_{j,i} x_i \right) \right),
\end{align*}
where $\mathbf{x} \in \mathbb{R}^{401}$ is a single image (input) and $y_{\kappa} \in [0,\,\,1]$ for ${\kappa}=0,\ldots,9$, can be interpreted as the conditional probability that the image contains the digit ${\kappa}$ given the input. Finally, the sigmoid function is given as $h(a) = \frac{1}{1+\exp{(-a)}}$. Let $\mathbf{y}^* = \begin{bmatrix} y^*_0, \ldots,y^*_{\kappa},\ldots, y^*_9\end{bmatrix}^\top$ denote the true class or label associated with input image $\mathbf{x}$ (in machine learning community, $\mathbf{y}^*$ is known as the target class or label). For example, if the image $\mathbf{x}$ contains the digit $9$, then $\mathbf{y}^* = \begin{bmatrix} \mathbf{0}_{1\times9} & 1\end{bmatrix}^T$. The conditional distribution of all target classes given inputs can be modeled as (see equation 5.22 of~\cite{bishop2016pattern})
\begin{align*}
    p\left(\mathbf{y}^*|\mathbf{x},\mathbf{w}\right) = \prod_{{\kappa}=0}^9\,y_{\kappa}\left(\mathbf{x},\mathbf{w}\right)^{y^*_{\kappa}}\left(1-y_{\kappa}\left(\mathbf{x},\mathbf{w}\right)\right)^{1-y^*_{\kappa}}.
\end{align*}
Taking the negative logarithm of the corresponding likelihood function yields the following empirical risk function:
\begin{align*}
\bar{R}(\mathbf{w}) &= -\sum_{j=1}^{2500}\,\sum_{{\kappa}=0}^9 \left( y^*_{j{\kappa}} \ln{\left(y_{\kappa}\left(\mathbf{x}_j,\mathbf{w}\right)\right)} \right.\\&\qquad \qquad + \left.(1-y^*_{j{\kappa}}) \ln{\left(1-y_{\kappa}\left(\mathbf{x}_j,\mathbf{w}\right)\right)}\right),
\end{align*}
where $y^*_{j{\kappa}}$ denotes the ${\kappa}$-th entry of $\mathbf{y}^*_j$ and $\mathbf{y}^*_j$ denotes the target class associated with input image $\mathbf{x}_j$. During training, each agent exchanges the weights $W^{(1)}$ and $W^{(2)}$ with its neighbors as described in the proposed algorithm. Here we conduct the following three experiments: (i) centralized SGD, where a centralized version of the SGD is implemented by a central node having all 2500 training data, (ii) a distributed SGD depicted in Fig.~\ref{Fig:NNPic} with equally distributed data, where 10 agents distributedly train 10 different neural nets, and (iii) a distributed SGD with class-specific data distributed among the agents. For experiment (ii), each node received 250 training data, randomly sampled from the entire training set, i.e., $m_i=250$ for all $i =1,\ldots,10$. For experiment (iii), data are distributed such that each agent only receives images corresponding to a particular class, i.e., agent $1$ received all the images of $0$s, agent $2$ received all the images of $1$s, and so forth. Thus for experiment (iii), we have $m_1 = 257$, $m_2 = 235$, $m_3 = 257$, $m_4 = 244$, $m_5 = 242$, $m_6 = 255$, $m_7 = 244$, $m_8 = 259$, $m_9 = 245$, and $m_{10} = 262$. For all three experiments, we select $\alpha_k = \displaystyle \frac{1}{(\varepsilon k + 1)}$, where $\varepsilon = 10^{-5}$. For experiments (ii) and (iii), we select $\beta_k = \displaystyle\frac{b}{(\varepsilon k + 1)^{1/3}}$, where $b = 0.2525$. Note that using a scale factor $\varepsilon$ does not affect the theoretical results provided in the previous sections.
%%%%%%%%%%%%%%%%%%%%%%%%%%%%%%%%%%%%%%%%%%%%%%%%%%%%%%%%%%%%%%%%%%%%%%%%%%%%
\begin{figure}[h]
  \begin{centering}
      \subfigure[Experiment (i)]{
      \psfrag{k}[]{\footnotesize{$k$}}
      \psfrag{J(k)}[][]{\scriptsize{$\bar{R}(\mathbf{w}_k)$}}
      \psfrag{Centralized SGD}[][]{\scriptsize{Centralized SGD }}
      \includegraphics[scale=0.27]{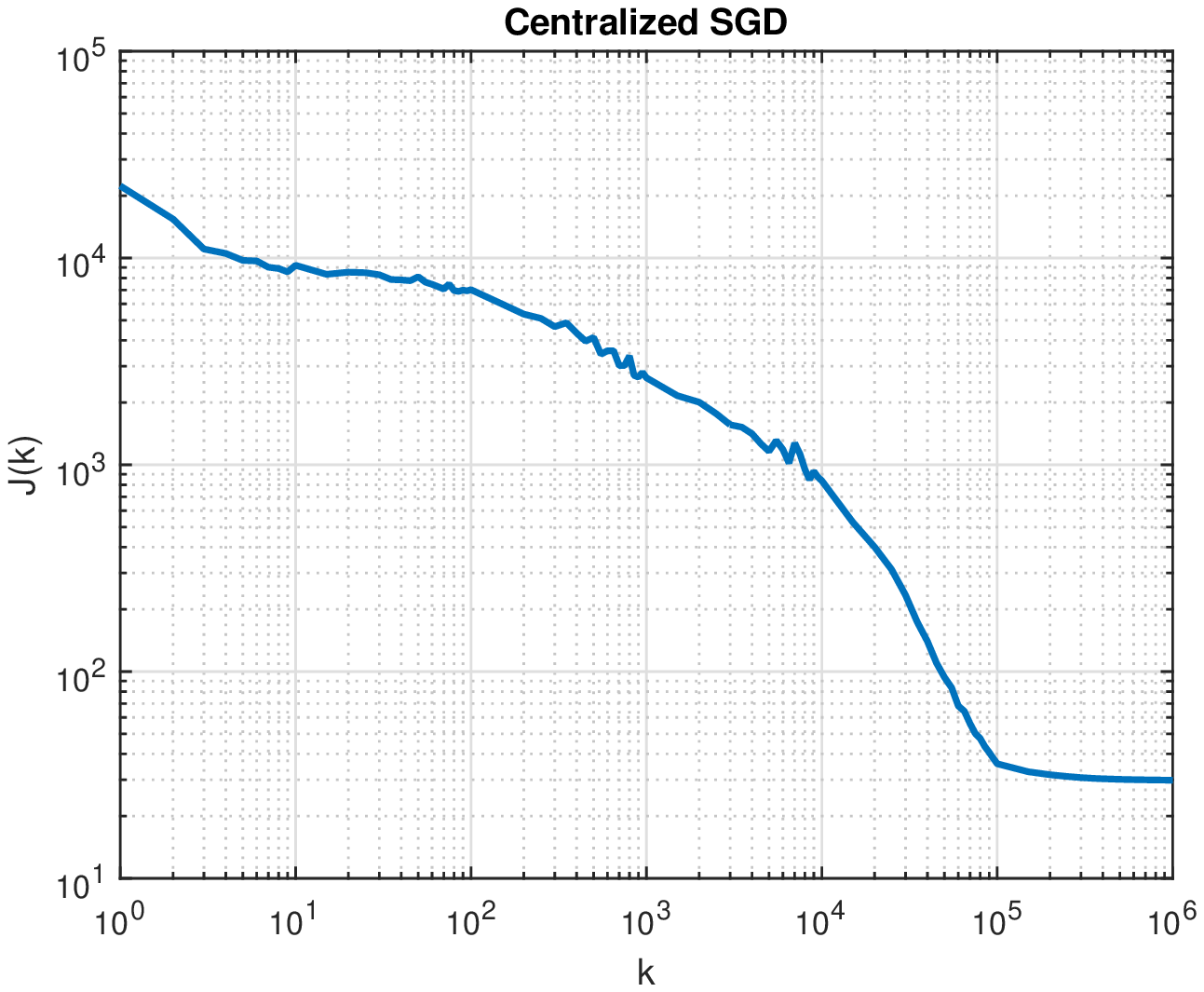}\label{Fig:Cent}}
      \subfigure[Experiment (ii)]{
      \psfrag{k}[][]{\footnotesize{$k$}}
      \psfrag{J(k)}[][]{\scriptsize{$\bar{R}(\mathbf{w}_k)$}}
      \psfrag{Distributed SGD}[][]{\scriptsize{Distributed SGD}}
      \includegraphics[scale=0.29]{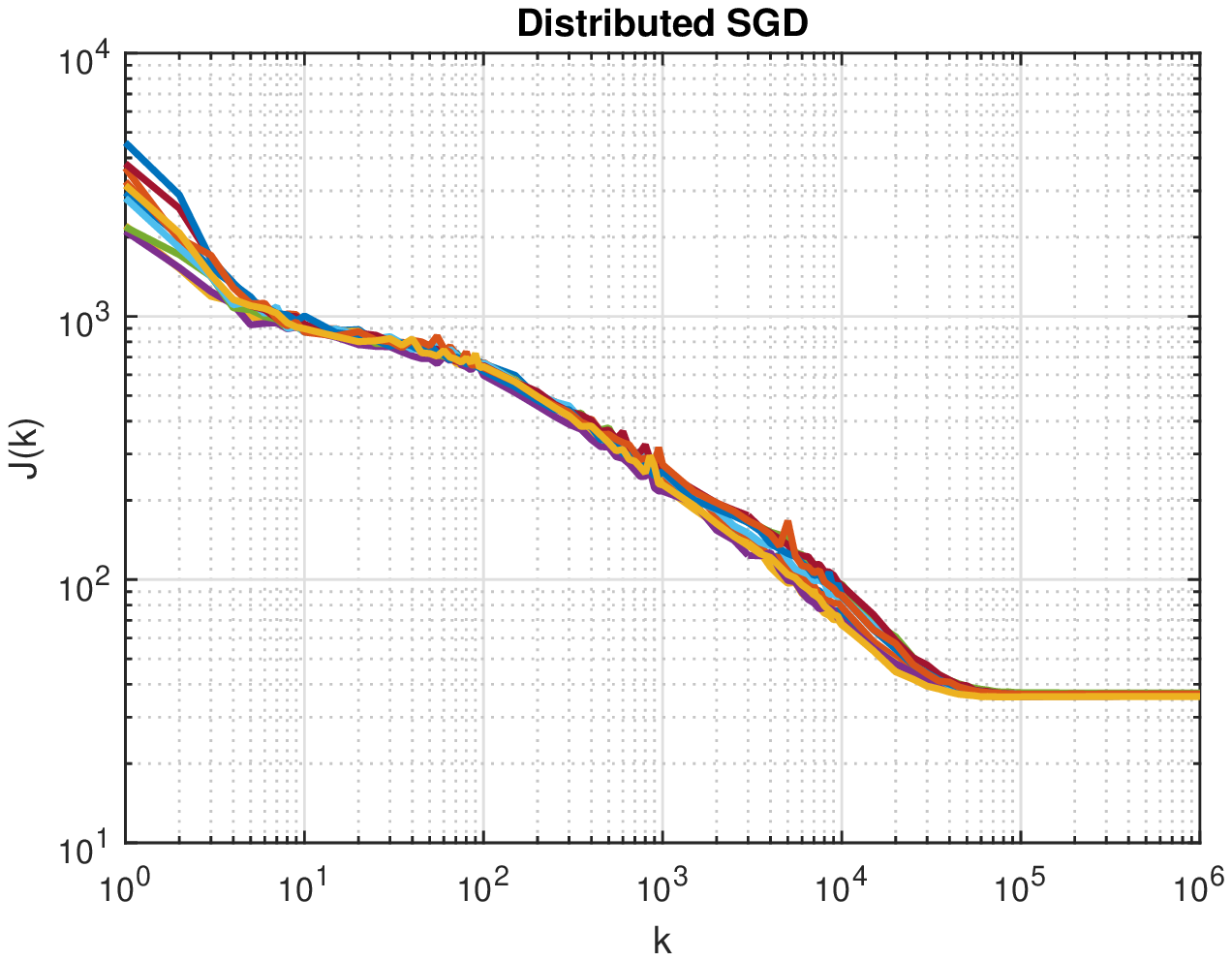}\label{Fig:Dist}}
      \subfigure[Experiment (iii)]{
      \psfrag{k}[][]{\footnotesize{$k$}}
      \psfrag{J(k)}[][]{\scriptsize{$\bar{R}(\mathbf{w}_k)$}}
      \psfrag{Distributed SGD}[][]{\scriptsize{Distributed SGD}}
      \includegraphics[scale=0.27]{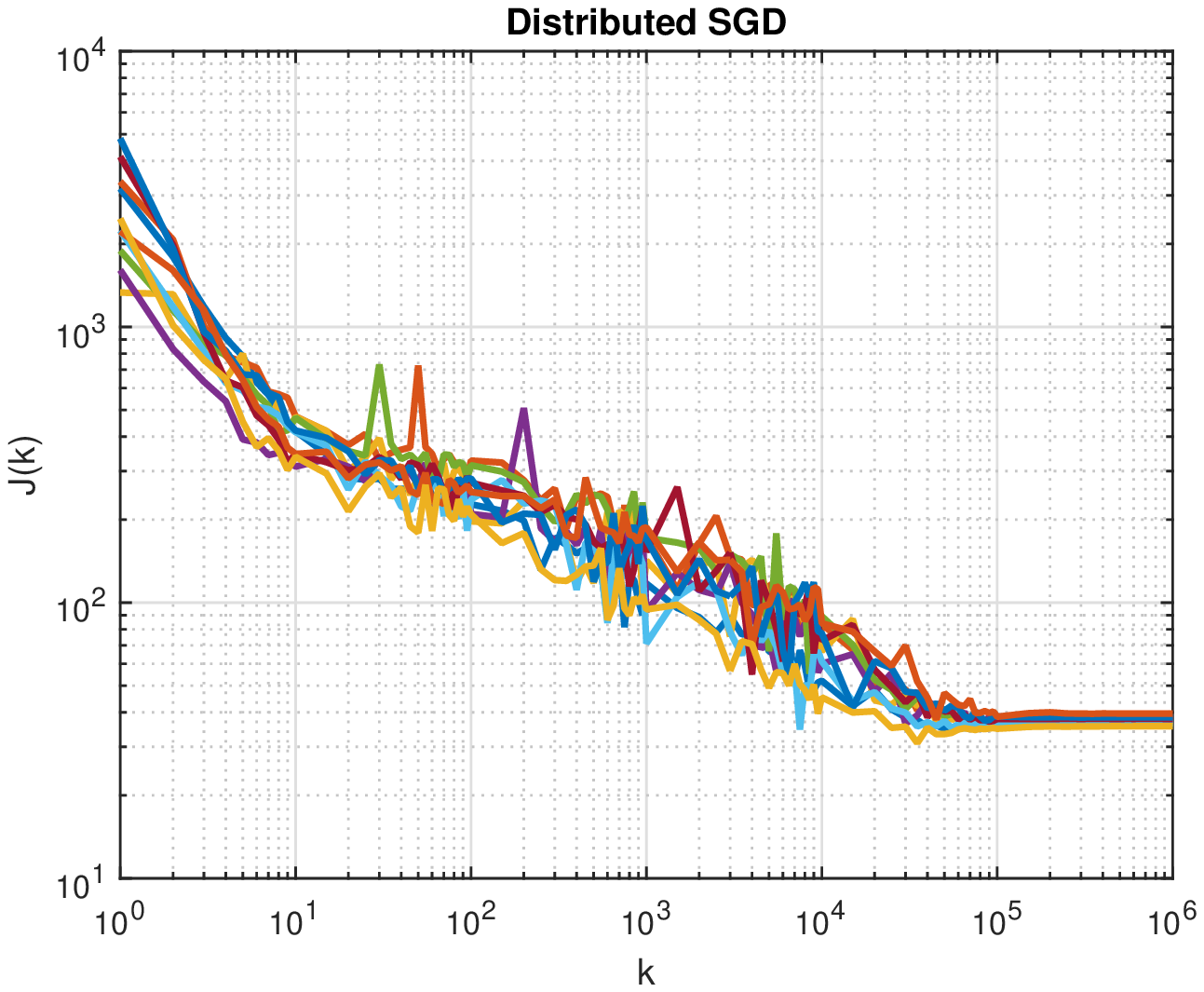}\label{Fig:Dist1}}
      \subfigure[Misclassification examples]{
      \psfrag{7 - 5}[][]{\scriptsize{$7\rightarrow5$}}
      \psfrag{2 - 4}[][]{\scriptsize{$2\rightarrow4$}}
      \psfrag{3 - 7}[][]{\scriptsize{$3\rightarrow7$}}
      \psfrag{4 - 9}[][]{\scriptsize{$4\rightarrow9$}}
      \psfrag{9 - 7}[][]{\scriptsize{$9\rightarrow7$}}
      \psfrag{4 - 9}[][]{\scriptsize{$4\rightarrow9$}}
      \psfrag{5 - 9}[][]{\scriptsize{$5\rightarrow9$}}
      \psfrag{5 - 0}[][]{\scriptsize{$5\rightarrow0$}}
      \psfrag{9 - 3}[][]{\scriptsize{$9\rightarrow3$}}
      \includegraphics[scale=0.29]{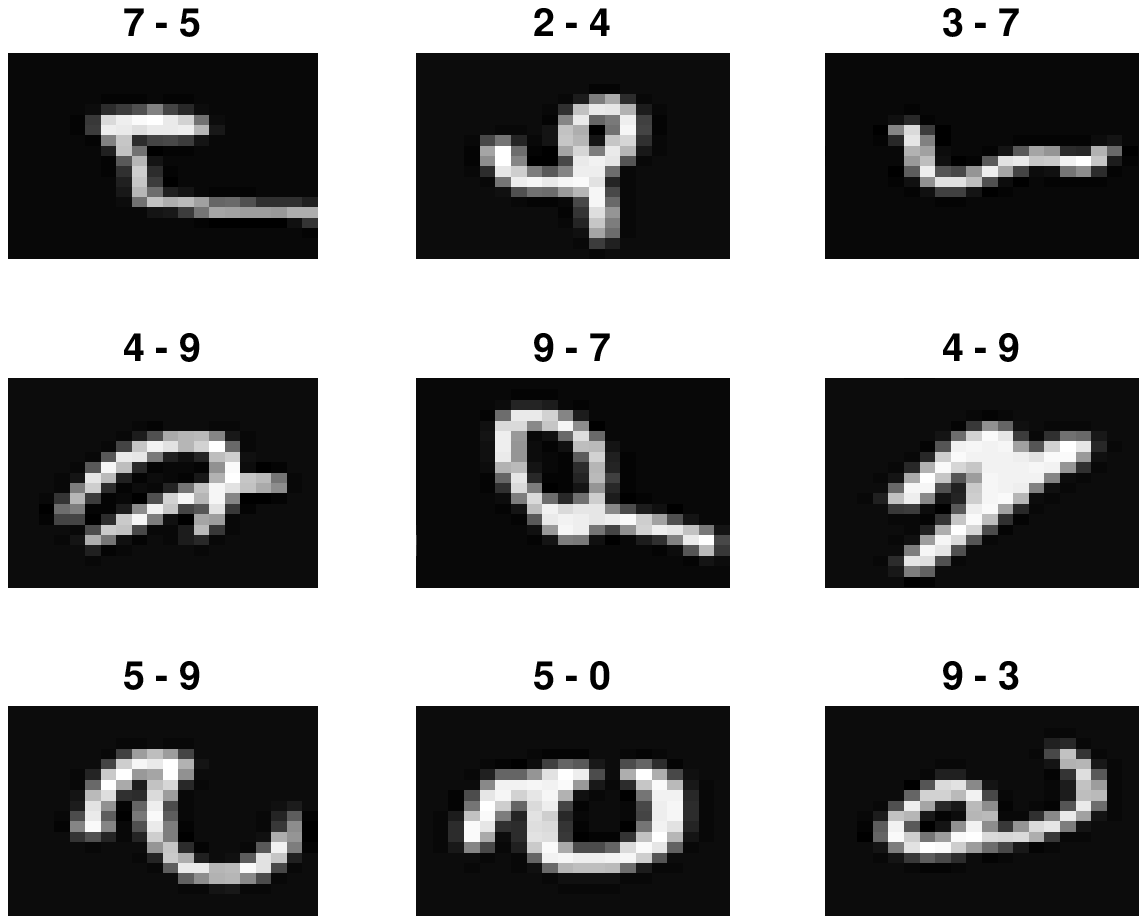}\label{Fig:Mistakes}}
      \caption{Empirical risk for all three experiments and a few misclassification examples.}
      \label{Plots}
  \end{centering}
\end{figure}
%%%%%%%%%%%%%%%%%%%%%%%%%%%%%%%%%%%%%%%%%%%%%%%%%%%%%%%%%%%%%%%%%%%%%%%%%%%%

Given in Fig.~\ref{Plots} are the results obtained from the three experiments. The risks obtained from experiments (i), (ii), and (iii) are given in Figs.~\ref{Fig:Cent}, \ref{Fig:Dist}, and \ref{Fig:Dist1}, respectively. For all three experiments, the error rate, i.e., \% of images misclassified, obtained from running the trained neural net on the testing data of 2500 images are
%\begin{align*}
%  \textnormal{Experiment (i):} \qquad &\textnormal{error \%:\,\,} 7.12 \\
%  \textnormal{Experiment (ii):} \qquad &\textnormal{error \%:\,\,} 7.36 \\
%  \textnormal{Experiment (iii):} \qquad &\textnormal{error \%:\,\,} 7.36
%\end{align*}
\begin{align*}
  \textnormal{Experiments (i): 7.12\%,\,\,\,\,(ii): 7.36\%,\,\,\,\,(iii): 7.36\%}
\end{align*}
Finally, a few misclassification examples are given in Fig.~\ref{Fig:Mistakes}, where a 7 is misclassified as a 5, 2 as a 4, and so forth. Results given here indicate that regardless of how the data are distributed, the agents are able to train their network and the distributedly trained networks are able to yield similar performance as that of a centrally trained network. More importantly, in experiment (iii), agents were able to recognize all 10 classes even though they only had access to data corresponding to a single class. This result has numerous implications for the machine learning community, specifically for federated multi-task learning under information flow constraints.

\section{Conclusion}\label{sec:conclusion}

This paper presented the development of a distributed stochastic gradient descent algorithm for solving non-convex optimization problems. Here we assumed that the local objective functions are Lipschitz continuous and twice continuously differentiable with Lipschitz continuous gradients and Hessians. We provided sufficient conditions on algorithm step-sizes that guarantee asymptotic mean-square convergence of the proposed algorithm to a critical point. We applied the developed algorithm to a distributed supervised-learning problem, in which a set of 10 networked agents collaboratively train their individual neural nets to recognize handwritten digits in images. Results indicate that regardless of how the data are distributed, the agents are able to train their network and the distributedly trained networks are able to yield similar performance as that of a centrally trained network. Numerical results also show that the proposed distributed algorithm allowed individual agents to collaboratively recognize all 10 classes even though they only had access to data corresponding to a single class.

\section*{Appendix}

\subsection{Useful Lemmas}\label{Appdx:A}

%~~~~~~~~~~~~~~~~~~~~~~~~~~~~~~~~~~~~~~~~~~~~~~~~~~~~~~~~~~~~~~~~~~~~~~~~~~~~~~~~~~~~~~~~~~~~~~~~~~~~~~~~~~~~~~~~~~~~~~~~~~~~~~~~~~
\begin{lemma}\label{Lemma:Kar}
  Let $\{z_k\}$ be a non-negative sequence satisfying
\begin{align}\label{Eqn:lamma2}
z_{k+1} \leq \left( 1 - r_1(k)\right)z_k + r_2(k),
\end{align}
where $\{r_1(k)\}$ and $\{r_2(k)\}$ are sequences with
\begin{equation}
    \frac{a_1}{(k+1)^{\epsilon_1}} \leq r_1(k) \leq 1 \quad \textnormal{and} \quad r_2(k) \leq \frac{a_2}{(k+1)^{\epsilon_2}},
\end{equation}
where $0 < a_1$, $0 < a_2$, $0 \leq \epsilon_1 < 1$, and $\epsilon_1 < \epsilon_2$. Then $ (k+1)^{\epsilon_0}z_k \rightarrow 0$ as $k\rightarrow \infty$ for all $0 \leq \epsilon_0 < \epsilon_2 - \epsilon_1$.
\end{lemma}
%~~~~~~~~~~~~~~~~~~~~~~~~~~~~~~~~~~~~~~~~~~~~~~~~~~~~~~~~~~~~~~~~~~~~~~~~~~~~~~~~~~~~~~~~~~~~~~~~~~~~~~~~~~~~~~~~~~~~~~~~~~~~~~~~~~
\begin{proof}
This Lemma follows directly from Lemma 4.1 of \cite{Kar2013SIAM}.
\end{proof}

%~~~~~~~~~~~~~~~~~~~~~~~~~~~~~~~~~~~~~~~~~~~~~~~~~~~~~~~~~~~~~~~~~~~~~~~~~~~~~~~~~~~~~~~~~~~~~~~~~~~~~~~~~~~~~~~~~~~~~~~~~~~~~~~~~~
\begin{lemma}\label{Lemma:Robbins}
  Let $\{v_k\}$ be a non-negative sequence for which the following relation hold for all $k \geq 0$:
  \begin{equation}\label{Eq:Robbins}
      v_{k+1} \leq (1 + a_k) v_k - u_k + w_k,
  \end{equation}
  where $a_k\geq 0$, $u_k\geq 0$ and $w_k\geq 0$ with $\sum\limits_{k=0}^{\infty}a_k < \infty $ and $\sum\limits_{k=0}^{\infty}w_k < \infty $. Then the sequence $\{v_k\}$ will converge to $v\geq0$ and we further have $\sum\limits_{k=0}^{\infty}u_k < \infty $.
\end{lemma}
%~~~~~~~~~~~~~~~~~~~~~~~~~~~~~~~~~~~~~~~~~~~~~~~~~~~~~~~~~~~~~~~~~~~~~~~~~~~~~~~~~~~~~~~~~~~~~~~~~~~~~~~~~~~~~~~~~~~~~~~~~~~~~~~~~~
\begin{proof}
See \cite{ROBBINS1971233}.
\end{proof}
%~~~~~~~~~~~~~~~~~~~~~~~~~~~~~~~~~~~~~~~~~~~~~~~~~~~~~~~~~~~~~~~~~~~~~~~~~~~~~~~~~~~~~~~~~~~~~~~~~~~~~~~~~~~~~~~~~~~~~~~~~~~~~~~~~~
\begin{lemma}\label{Lemma10}
  Let $\gamma_k \triangleq \displaystyle\frac{a/b}{(k+1)^{\epsilon}}$ with $0 < \epsilon \leq 1$. Then it holds
  \begin{equation}
      \gamma_{k+1}^{-1} - \gamma_{k}^{-1} \leq \frac{2b\epsilon}{a} (k+1)^{\epsilon-1}.
  \end{equation}
\end{lemma}
%~~~~~~~~~~~~~~~~~~~~~~~~~~~~~~~~~~~~~~~~~~~~~~~~~~~~~~~~~~~~~~~~~~~~~~~~~~~~~~~~~~~~~~~~~~~~~~~~~~~~~~~~~~~~~~~~~~~~~~~~~~~~~~~~~~
\begin{proof}
This Lemma is a direct consequence of Lemma 10 of \cite{Zeng2018TSP}.
%First note that $(1+x)^{\epsilon}$ is a monotonically increasing function for all $x\in[0,\,1]$. Thus the following inequality holds for all $ \epsilon\in(0,\,1]$
%\begin{equation}\label{Eq:Lemma10}
%    (1+x)^{\epsilon} - 1 \leq 2\epsilon x, \quad \forall x\in[0,\,1].
%\end{equation}
%Note that
%\begin{align*}
%    \gamma_{k+1}^{-1} - \gamma_{k}^{-1} &= \frac{b}{a} \left( (k+2)^{\epsilon} - (k+1)^{\epsilon} \right)\\
%    &= \frac{b}{a} (k+1)^{\epsilon} \left( \left( 1 + \frac{1}{k+1} \right)^{\epsilon} - 1 \right)
%\end{align*}
%From \eqref{Eq:Lemma10} we have
%\begin{align*}
%    \left( 1 + \frac{1}{k+1} \right)^{\epsilon} - 1 \leq 2\epsilon \frac{1}{k+1}
%\end{align*}
%Therefore
%\begin{align*}
%    \gamma_{k+1}^{-1} - \gamma_{k}^{-1} \leq \frac{2b\epsilon}{a} (k+1)^{\epsilon-1}.
%\end{align*}
\end{proof}

\subsection{Proof of Theorem \ref{Theorem:Consensus}}\label{Appdx:B}
Define the average-consensus error as $\tilde{\mathbf{w}}_k = \left(M \otimes I_{d_w}\right)\mathbf{w}_k$, where $M = I_{n} - \frac{1}{n}\mathbf{1}_{n}\mathbf{1}_{n}^\top$. Thus from \eqref{Eq:DSG1} we have
\begin{equation*}
    \tilde{\mathbf{w}}_{k+1} = \left(\mathcal{W}_k \otimes I_{d_w} \right) \tilde{\mathbf{w}}_k - \alpha_k \left(M \otimes I_{d_w}\right) \mathbf{g}(\mathbf{w}_k,\bm{\xi}_k)
\end{equation*}
and $ \| \tilde{\mathbf{w}}_{k+1} \|_2 \leq \| \left( \left(I_{n}-\beta_k\mathcal{L}\right) \otimes I_{d_w}\right)  \tilde{\mathbf{w}}_k \|_2 + \alpha_k \| \left(M \otimes I_{d_w}\right) \|_2 \| \mathbf{g}(\mathbf{w}_k,\bm{\xi}_k) \|_2.$
Since $\mathbf{1}_{nd_w}^\top \tilde{\mathbf{w}}_k = 0$, it follows from Lemma 4.4 of \cite{Kar2013SIAM} that
\begin{align*}
    \| \left( \left(I_{n}-\beta_k\mathcal{L}\right) \otimes I_{d_w}\right)  \tilde{\mathbf{w}}_k \|_2 \leq (1 - \beta_k\lambda_2(\mathcal{L})) \| \tilde{\mathbf{w}}_k \|_2,
\end{align*}
where $\lambda_2(\cdot)$ denotes the second smallest eigenvalue. Thus we have
\begin{align*}
    \| \tilde{\mathbf{w}}_{k+1} \|_2
    &\leq  (1 - \beta_k\lambda_2(\mathcal{L})) \| \tilde{\mathbf{w}}_k \|_2 + \alpha_k \| \mathbf{g}(\mathbf{w}_k,\bm{\xi}_k) \|_2.
\end{align*}
Now we use the following inequality
\begin{equation}\label{InEqTrick}
    (x+y)^2 \leq (1+\theta)x^2 + \left( 1+ \frac{1}{\theta}\right)y^2,
\end{equation}
for all $x,y,\in\mathbb{R}$ and $\theta > 0$. Selecting $\theta = \beta_k\lambda_2(\mathcal{L})$ yields
\begin{align*}
    &\| \tilde{\mathbf{w}}_{k+1} \|_2^2
    \leq  (1 + \beta_k\lambda_2(\mathcal{L})) (1 - \beta_k\lambda_2(\mathcal{L}))  \| \tilde{\mathbf{w}}_k \|_2^2  \\
    &\qquad \qquad \qquad \qquad + \alpha_k^2 \left( 1 + \frac{1}{\beta_k\lambda_2(\mathcal{L})} \right)  \| \mathbf{g}(\mathbf{w}_k,\bm{\xi}_k) \|_2^2 \\
    &\,\,= (1 - \beta^2_k\lambda_2(\mathcal{L})^2)  \| \tilde{\mathbf{w}}_k \|_2^2  + \alpha_k^2 \left( \frac{1 + \beta_k\lambda_2(\mathcal{L})}{\beta_k\lambda_2(\mathcal{L})} \right)  \| \mathbf{g}(\mathbf{w}_k,\bm{\xi}_k) \|_2^2
\end{align*}
Now taking the expectation yields
\begin{align}\label{wTildeE2}
\begin{split}
   \mathbb{E} \left[ \| \tilde{\mathbf{w}}_{k+1} \|_2^2 \right]
    &\leq   (1 - \beta^2_k\lambda_2(\mathcal{L})^2) \mathbb{E} \left[ \| \tilde{\mathbf{w}}_k \|_2^2 \right]
    \\& + \alpha_k^2 \left( \frac{1 + \beta_k\lambda_2(\mathcal{L})}{\beta_k\lambda_2(\mathcal{L})} \right) \mathbb{E} \left[ \| \mathbf{g}(\mathbf{w}_k,\bm{\xi}_k) \|_2^2 \right]
\end{split}
\end{align}
Using Proposition \ref{Assump:BoundedGrad}, \eqref{wTildeE2} can be written as
\begin{align}\label{wTildeE2a}
\begin{split}
   \mathbb{E} \left[ \| \tilde{\mathbf{w}}_{k+1} \|_2^2 \right]
    &\leq (1 - \beta^2\lambda_2(\mathcal{L})^2) \mathbb{E} \left[ \| \tilde{\mathbf{w}}_k \|_2^2 \right] \\
    &\qquad + \frac{\alpha_k^2}{\beta_k} \left( \frac{ \left(1 + \beta_k\lambda_2(\mathcal{L})\right)\mu_g}{\lambda_2(\mathcal{L})} \right)
\end{split}
\end{align}
Note $\left( \displaystyle \frac{ \left(1 + \beta_k\lambda_2(\mathcal{L})\right)\mu_g}{\lambda_2(\mathcal{L})} \right) \leq \left( \displaystyle \frac{ \left(1 + b\lambda_2(\mathcal{L})\right)\mu_g}{\lambda_2(\mathcal{L})} \right) \triangleq \mu_a,$ for some $\mu_a > 0$. Let $r_1(k) = \beta^2_k\lambda_2(\mathcal{L})^2 = \displaystyle\frac{b^2\lambda_2(\mathcal{L})^2}{(k+1)^{2\delta_1}}$ and $r_2(k) = \frac{\alpha_k^2}{\beta_k} \left( \frac{ \left(1 + \beta_k\lambda_2(\mathcal{L})\right)\mu_g}{\lambda_2(\mathcal{L})} \right) \leq \frac{a^2\mu_a/b}{(k+1)^{2\delta_2-\delta_1}}$. Now \eqref{wTildeE2a} can be written in the form of \eqref{Eqn:lamma2} with  $\epsilon_1 = 2\delta_1$ and $\epsilon_2 = 2\delta_2-\delta_1$. Thus it follows from Lemma \ref{Lemma:Kar} that
\begin{align*}
    (k+1)^{\delta_0}\, \mathbb{E} \left[ \| \tilde{\mathbf{w}}_k \|_2^2 \right] \rightarrow 0
    \quad \textnormal{as}\quad k\rightarrow \infty, \quad \forall\,0 \leq \delta_0 < 2\delta_2-3\delta_1.
\end{align*}
Thus there exists a constant $0 < \mu_{w} < \infty$ such that for all $k\geq 0$
\begin{align}\label{Eq:MSbound1}
    \mathbb{E} \left[ \| \tilde{\mathbf{w}}_k \|_2^2 \right] \leq \mu_w \frac{1}{(k+1)^{\delta_0}}, \quad \forall\,0 \leq \delta_0 < 2\delta_2-3\delta_1.
\end{align}
Now \eqref{Eq:MSbound1a} follows from Assumption~\ref{Assump:AlphaBeta} that $\delta_2 > 3\delta_1$.
% and
%\begin{align*}
%    \mathbb{E} \left[ \| \tilde{\mathbf{w}}_k \|_2^2 \right] = O\left(\displaystyle\frac{1}{(k+1)^{\delta_2}}\right).
%\end{align*}

\subsection{Proof of Theorem \ref{Theorem:InfSum}}\label{Appdx:C}
From \eqref{dVk} we have
\begin{align*}
&\nabla V(\gamma_k,\mathbf{w}_{k+1}) - \nabla V(\gamma_k,\mathbf{w}_k) = \nabla F(\,\mathbf{w}_{k+1}\,) - \nabla F(\,\mathbf{w}_k\,)  \\
&\qquad \qquad + \frac{1}{\gamma_k}  \left(\mathcal{L}\otimes I_{d_w}\right)\mathbf{w}_{k+1} -  \frac{1}{\gamma_k} \left(\mathcal{L}\otimes I_{d_w}\right)\mathbf{w}_k
\end{align*}
Now based on Assumption \ref{Assump:Lipz}, for a fixed $\gamma_k$, $\nabla V(\gamma_k,\,\mathbf{w}\,)$ is Lipschitz continuous in $\mathbf{w}$. Thus we have
\begin{align}
\begin{split}
  \| \nabla V(\gamma_k,\mathbf{w}_{k+1}) - &\nabla V(\gamma_k,\mathbf{w}_k)\|_2 \\
  %&\leq \| \nabla F(\,\mathbf{w}_{k+1}\,) - \nabla F(\,\mathbf{w}_k\,) \|_2 + \frac{1}{\gamma_k}\, \sigma_{\max}(\mathcal{L}) \| \mathbf{w}_{k+1} - \mathbf{w}_{k} \|_2\\
&\quad \leq  \left( L + \frac{\sigma_{\max}(\mathcal{L})}{\gamma_k} \right) \|\mathbf{w}_{k+1}-\mathbf{w}_k\|_2 \label{Eq:dV0}
\end{split}
\end{align}
It follows from Lemma~\ref{Lemma:Lipz} that
\begin{align}
%V(\gamma_k,\mathbf{w}_{k+1}) - V(\gamma_k,\mathbf{w}_k) &\leq \nabla V(\gamma_k,\mathbf{w}_k)^\top \left( \mathbf{w}_{k + 1} - \mathbf{w}_{k} \right)  + \frac{1}{2}\left( L + \frac{\sigma_{\max}(\mathcal{L})}{\gamma_k} \right) \| \mathbf{w}_{k + 1} - \mathbf{w}_{k} \|_2^2\\
\begin{split}
&V(\gamma_k,\mathbf{w}_{k+1}) - V(\gamma_k,\mathbf{w}_k)
  \leq \frac{1}{2}\left( L + \frac{\sigma_{\max}(\mathcal{L})}{\gamma_k} \right) \| \mathbf{w}_{k + 1} - \mathbf{w}_{k} \|_2^2  \\
  &+ \qquad \left( \nabla F(\mathbf{w}_k) + \frac{1}{\gamma_k} \left(\mathcal{L}\otimes I_{d_w}\right)\mathbf{w}_k \right)^\top \left( \mathbf{w}_{k + 1} - \mathbf{w}_{k} \right)
\end{split}\label{Eq:dV}
\end{align}
Note that the distributed SGD algorithm in \eqref{Eq:DSG1} can be rewritten as
\begin{align}\label{Eq:DSG2}
  \mathbf{w}_{k + 1} - \mathbf{w}_{k} &= - \alpha_k \left( \mathbf{g}(\mathbf{w}_k,\bm{\xi}_k) + \frac{1}{\gamma_k} \left(\mathcal{L}\otimes I_{d_w}\right)\mathbf{w}_k \right).
\end{align}
Substituting \eqref{Eq:DSG2} into \eqref{Eq:dV} and taking the conditional expectation $\mathbb{E}_{\xi}\left[ \,\cdot\, \right]$ yields
%\begin{align}
%\begin{split}
%  V(\gamma_k,\mathbf{w}_{k+1}) - V(\gamma_k,\mathbf{w}_k) &\leq - \alpha_k \left( \nabla F(\mathbf{w}_k) + \frac{1}{\gamma_k} \left(\mathcal{L}\otimes I_{d_w}\right)\mathbf{w}_k \right)^\top \left( \mathbf{g}(\mathbf{w}_k,\bm{\xi}_k) + \frac{1}{\gamma_k} \left(\mathcal{L}\otimes I_{d_w}\right)\mathbf{w}_k \right)  \\
%  &\qquad\qquad + \frac{\alpha_k^2}{2}\left( L + \frac{\sigma_{\max}(\mathcal{L})}{\gamma_k} \right) \| \mathbf{g}(\mathbf{w}_k,\bm{\xi}_k) + \frac{1}{\gamma_k} \left(\mathcal{L}\otimes I_{d_w}\right)\mathbf{w}_k \|_2^2
%\end{split}\label{Eq:dV0}
%\end{align}
%Now
\begin{align}
\begin{split}
  &\mathbb{E}_{\xi}\left[V(\gamma_k,\mathbf{w}_{k+1})\right] - V(\gamma_k,\mathbf{w}_k) \leq   \frac{\alpha_k^2}{2}\left( L + \frac{\sigma_{\max}(\mathcal{L})}{\gamma_k} \right)\\
  &\quad\times \mathbb{E}_{\xi}\left[ \| \mathbf{g}(\mathbf{w}_k,\bm{\xi}_k) + \frac{1}{\gamma_k} \left(\mathcal{L}\otimes I_{d_w}\right)\mathbf{w}_k \|_2^2 \right]\\
  & - \alpha_k \left( \nabla F(\mathbf{w}_k) + \frac{1}{\gamma_k} \left(\mathcal{L}\otimes I_{d_w}\right)\mathbf{w}_k \right)^\top \\
  &\quad \times\left( \mathbb{E}_{\xi}\left[ \mathbf{g}(\mathbf{w}_k,\bm{\xi}_k) \right] + \frac{1}{\gamma_k} \left(\mathcal{L}\otimes I_{d_w}\right)\mathbf{w}_k \right) \quad \textnormal{a.s.}
\end{split}\label{Eq:dV0}
\end{align}
Based on Assumption \ref{Assump:Grad1}, there exists $\mu > 0$ such that
\begin{align*}
  &\left( \nabla F(\mathbf{w}_k) + \frac{1}{\gamma_k} \left(\mathcal{L}\otimes I_{d_w}\right)\mathbf{w}_k \right)^\top  \bigg(  \mathbb{E}_{\xi}\left[\, \mathbf{g}(\mathbf{w}_k,\bm{\xi}_k) \,\right] + \\
  &\left.\frac{1}{\gamma_k} \left(\mathcal{L}\otimes I_{d_w}\right)\mathbf{w}_k \right) \geq \mu \left\|  \nabla F(\mathbf{w}_k) + \frac{1}{\gamma_k} \left(\mathcal{L}\otimes I_{d_w}\right)\mathbf{w}_k \right\|^2_2,\,\, \textnormal{a.s.}
  \end{align*}
Thus we have
\begin{align}
\begin{split}
  &\mathbb{E}_{\xi}\left[V(\gamma_k,\mathbf{w}_{k+1})\right] - V(\gamma_k,\mathbf{w}_k) \leq   \frac{\alpha_k^2}{2}\left( L + \frac{\sigma_{\max}(\mathcal{L})}{\gamma_k} \right)\\
  &\quad\times \mathbb{E}_{\xi}\left[ \| \mathbf{g}(\mathbf{w}_k,\bm{\xi}_k) + \frac{1}{\gamma_k} \left(\mathcal{L}\otimes I_{d_w}\right)\mathbf{w}_k \|_2^2 \right]\\
  & - \alpha_k \mu \left\|  \nabla F(\mathbf{w}_k) + \frac{1}{\gamma_k} \left(\mathcal{L}\otimes I_{d_w}\right)\mathbf{w}_k \right\|^2_2\,\, \textnormal{a.s.}
\end{split}\label{Eq:dV1}
\end{align}
Let
\begin{equation}\label{ck}
  c_k \triangleq \left( \alpha_k L + \sigma_{\max}(\mathcal{L})\beta_k \right).
\end{equation}
Now \eqref{Eq:dV1} can be written as
\begin{align}
\begin{split}
  &\mathbb{E}_{\xi} \left[ V(\gamma_k,\mathbf{w}_{k+1}) \right] - V(\gamma_k,\mathbf{w}_k) \leq \frac{1}{2} \alpha_k c_k \mathbb{E}_{\xi} \left[ \| \mathbf{g}(\mathbf{w}_k,\bm{\xi}_k) \right. \\
&\left.+ \frac{1}{\gamma_k} \left(\mathcal{L}\otimes I_{d_w}\right)\mathbf{w}_k \|_2^2 \right] - \alpha_k \mu \left\|  \nabla F(\mathbf{w}_k) + \frac{1}{\gamma_k} \left(\mathcal{L}\otimes I_{d_w}\right)\mathbf{w}_k \right\|^2_2
\end{split}\label{Eq:dV1a}
\end{align}
Based on Assumptions~\ref{Assump:Grad1} and \ref{Assump:Grad2}, there exists scalars ${\mu}_{v_1} \geq 0$ and ${\mu}_{v_2} \geq 0$ such that
\begin{align}
\begin{split}
&\mathbb{E}_{\xi} \left[ \| \mathbf{g}(\mathbf{w}_k,\bm{\xi}_k) + \frac{1}{\gamma_k} \left(\mathcal{L}\otimes I_{d_w}\right)\mathbf{w}_k \|_2^2 \right] \leq
\mu_{v_1} \\&+ \mu_{v_2} \left\|  \nabla F(\mathbf{w}_k) + \frac{1}{\gamma_k} \left(\mathcal{L}\otimes I_{d_w}\right)\mathbf{w}_k \right\|^2_2 \quad \textnormal{a.s.}
\end{split}\label{Eq:2ndGrad1}
\end{align}
Thus from \eqref{Eq:dV1a} we have
\begin{align}
\begin{split}
  &\mathbb{E}_{\xi} \left[ V(\gamma_k,\mathbf{w}_{k+1}) \right] - V(\gamma_k,\mathbf{w}_k) \leq \left( \frac{1}{2}c_k \mu_{v_2} -  \mu \right)\alpha_k \\
  &\times\left\|  \nabla F(\mathbf{w}_k) + \frac{1}{\gamma_k} \left(\mathcal{L}\otimes I_{d_w}\right)\mathbf{w}_k \right\|^2_2  + \frac{1}{2} c_k\alpha_k\mu_{v_1}\quad \textnormal{a.s.}
\end{split}\label{Eq:dV2}
\end{align}
Substituting $\nabla V(\gamma_k, \mathbf{w}_k) = \nabla F(\mathbf{w}_k) + \displaystyle\frac{1}{\gamma_k} \left(\mathcal{L}\otimes I_{d_w}\right)\mathbf{w}_k$ and taking the total expectation of \eqref{Eq:dV2} yields
\begin{align}
\begin{split}
  &\mathbb{E}\left[ V(\gamma_k,\mathbf{w}_{k+1}) \right] - \mathbb{E}\left[ V(\gamma_k,\mathbf{w}_k)\right] \leq - \left( \mu - \frac{1}{2} c_k \mu_{v_2} \right)\alpha_k \\
  &\times \mathbb{E}\left[ \left\|  \nabla V(\gamma_k, \mathbf{w}_k)  \right\|^2_2 \right] + \frac{1}{2}c_k \alpha_k \mu_{v_1}
\end{split}\label{Eq:dV3}
\end{align}
Note that
\begin{align}
\begin{split}
  &V(\gamma_{k+1},\mathbf{w}_{k+1}) = V(\gamma_k,\mathbf{w}_{k+1})\\ &\qquad + \frac{1}{2} \left( \gamma_{k+1}^{-1} - \gamma_{k}^{-1} \right) \mathbf{w}^\top_{k+1} \left(\mathcal{L}\otimes I_{d_w}\right)\mathbf{w}_{k+1}
\end{split}\label{Eq:deltaV}
\end{align}
Combining \eqref{Eq:dV3} and \eqref{Eq:deltaV} yields
\begin{align}
\begin{split}
  &\mathbb{E}\left[ V(\gamma_{k+1},\mathbf{w}_{k+1}) \right] - \mathbb{E}\left[ V(\gamma_k,\mathbf{w}_k)\right] \leq \\
  &- \left( \mu - \frac{1}{2} c_k \mu_{v_2} \right)\alpha_k  \mathbb{E}\left[ \left\|  \nabla V(\gamma_k, \mathbf{w}_k)  \right\|^2_2 \right]  + \frac{1}{2}c_k \alpha_k \mu_{v_1} \\
  & \qquad + \frac{1}{2} \left( \gamma_{k+1}^{-1} - \gamma_{k}^{-1} \right) \mathbb{E}\left[ \mathbf{w}^\top_{k+1} \left(\mathcal{L}\otimes I_{d_w}\right)\mathbf{w}_{k+1} \right]
\end{split}\label{Eq:dV4}
\end{align}
\noindent If we select $\epsilon = \delta_2-\delta_1$, it follows directly from Lemma~\ref{Lemma10} that
  \begin{equation}
      \gamma_{k+1}^{-1} - \gamma_{k}^{-1} \leq \frac{2b\left(\delta_2-\delta_1\right)}{a} (k+1)^{\delta_2-\delta_1-1}.
  \end{equation}
Note that from Lemma~\ref{Lemma2} we have
$    \mathbf{w}^\top_{k+1} \left(\mathcal{L}\otimes I_{d_w}\right)\mathbf{w}_{k+1} =
    \tilde{\mathbf{w}}^\top_{k+1} \left(\mathcal{L}\otimes I_{d_w}\right)\tilde{\mathbf{w}}_{k+1}
    \leq \sigma_{\max}\left( \mathcal{L} \right) \| \tilde{\mathbf{w}}_{k+1} \|_2^2$. Thus
\begin{align*}
    &\frac{1}{2} \left( \gamma_{k+1}^{-1} - \gamma_{k}^{-1} \right) \mathbb{E}\left[ \mathbf{w}^\top_{k+1} \left(\mathcal{L}\otimes I_{d_w}\right)\mathbf{w}_{k+1} \right]
    \leq  \frac{2b\left(\delta_2-\delta_1\right)}{a} \\
    &\times (k+1)^{\delta_2-\delta_1-1} \sigma_{\max}\left( \mathcal{L} \right)  \mathbb{E}\left[ \| \tilde{\mathbf{w}}_{k+1} \|_2^2 \right]
\end{align*}
We have established in \eqref{Eq:MSbound1} that for all $k\geq 0$
\begin{align}
    \mathbb{E} \left[ \| \tilde{\mathbf{w}}_k \|_2^2 \right] \leq \mu_w \frac{1}{(k+1)^{\delta_0}}, \quad \forall\,0 \leq \delta_0 < 2\delta_2-3\delta_1.
\end{align}
Therefore we have
\begin{align*}
    &\frac{1}{2} \left( \gamma_{k+1}^{-1} - \gamma_{k}^{-1} \right) \mathbb{E}\left[ \mathbf{w}^\top_{k+1} \left(\mathcal{L}\otimes I_{d_w}\right)\mathbf{w}_{k+1} \right]
    \leq \\ & \frac{2b\left(\delta_2-\delta_1\right)}{a} (k+1)^{\delta_2-\delta_1-1} \sigma_{\max}\left( \mathcal{L} \right)  \mu_w \frac{1}{(k+1)^{\delta_0}}
\end{align*}
Let $\mu_c = \frac{2b\left(\delta_2-\delta_1\right)}{a} \sigma_{\max}\left( \mathcal{L} \right)  \mu_w$. Now selecting $\delta_0 = 2\delta_2-3\delta_1-\varepsilon$, where $0 < \varepsilon \ll \delta_1 $, yields
\begin{align*}
    &\frac{1}{2} \left( \gamma_{k+1}^{-1} - \gamma_{k}^{-1} \right) \mathbb{E}\left[ \mathbf{w}^\top_{k+1} \left(\mathcal{L}\otimes I_{d_w}\right)\mathbf{w}_{k+1} \right]
    \leq  \\
    &\mu_c (k+1)^{\delta_2-\delta_1-1-2\delta_2+3\delta_1+\varepsilon} = \mu_c (k+1)^{-\delta_2+2\delta_1-1+\varepsilon}
\end{align*}
Thus if we select $\delta_1$ and $\delta_2$ such that $\delta_2 > 2\delta_1 + \varepsilon$, then we have
\begin{equation}
    \frac{1}{2} \left( \gamma_{k+1}^{-1} - \gamma_{k}^{-1} \right) \mathbb{E}\left[ \mathbf{w}^\top_{k+1} \left(\mathcal{L}\otimes I_{d_w}\right)\mathbf{w}_{k+1} \right]
    \leq  \mu_c \frac{1}{(k+1)^{1+\varepsilon_1}},
\end{equation}
where $\varepsilon_1 > 0$ and $\delta_2 - 2\delta_1 - \varepsilon = \varepsilon_1$. Now we can write \eqref{Eq:dV4} as
\begin{align}
\begin{split}
  &\mathbb{E}\left[ V(\gamma_{k+1},\mathbf{w}_{k+1}) \right] - \mathbb{E}\left[ V(\gamma_k,\mathbf{w}_k)\right] \leq - \left( \mu - \frac{1}{2} c_k \mu_{v_2} \right)\\
  &\times \alpha_k \mathbb{E}\left[ \left\|  \nabla V(\gamma_k, \mathbf{w}_k)  \right\|^2_2 \right] + \frac{1}{2}c_k \alpha_k \mu_{v_1} + \mu_c \frac{1}{(k+1)^{1+\varepsilon_1}}
\end{split}\label{Eq:dV5}
\end{align}
Since $c_k$ is decreasing to zero, for sufficiently large $k$, we have $c_k \mu_{v_2} < \mu$. Therefore $\left( \mu - \frac{1}{2} c_k \mu_{v_2} \right) > \frac{1}{2} \mu$ for sufficiently large $k$. Thus we have
\begin{align}
\begin{split}
  &\mathbb{E}\left[ V(\gamma_{k+1},\mathbf{w}_{k+1}) \right] - \mathbb{E}\left[ V(\gamma_k,\mathbf{w}_k)\right] \leq - \frac{1}{2} \mu\alpha_k\\
  &\times \mathbb{E}\left[ \left\|  \nabla V(\gamma_k, \mathbf{w}_k)  \right\|^2_2 \right] + \frac{1}{2}c_k \alpha_k \mu_{v_1}    +  \frac{\mu_c}{(k+1)^{1+\varepsilon_1}}
\end{split}\label{Eq:dV6}
\end{align}
Now \eqref{Eq:dV6} can be written in the form of \eqref{Eq:Robbins} after selecting $a_k = 0$,
\begin{align}
    w_k &= \frac{1}{2}c_k \alpha_k \mu_{v_1}   +  \frac{\mu_c}{(k+1)^{1+\varepsilon_1}},\\
    u_k &= \frac{1}{2} \mu\alpha_k \mathbb{E}\left[ \left\|  \nabla V(\gamma_k, \mathbf{w}_k)  \right\|^2_2 \right].
\end{align}
Note that here we have $a_k = 0$, $u_k\geq 0$ and $w_k\geq 0$ with $\sum\limits_{k=0}^{\infty}a_k < \infty $ and $\sum\limits_{k=0}^{\infty}w_k < \infty $. Note $c_k\alpha_k$ is summable because $\alpha_k\beta_k$ is summable and $\alpha_k$ is square-summable. Therefore from Lemma~\ref{Lemma:Robbins} we have $\mathbb{E}\left[ V(\gamma_k,\mathbf{w}_k)\right]$ is a convergent sequence and  $\sum\limits_{k=0}^{\infty} \,\frac{1}{2} \mu \alpha_k \mathbb{E}\left[ \left\|  \nabla V(\gamma_k,\mathbf{w}_k)  \right\|^2_2 \right] < \infty $.

\subsection{Proof of Theorem \ref{Theorem:Convergence}}\label{Appdx:D}
Note that
    \begin{align}\label{Eq:DSG3}
      \left\| \mathbf{w}_{k + 1} - \mathbf{w}_{k} \right\|_2^2 &= \alpha_k^2  \left\| \mathbf{g}(\mathbf{w}_k,\bm{\xi}_k) + \frac{1}{\gamma_k} \left(\mathcal{L}\otimes I_{d_w}\right)\mathbf{w}_k \right\|_2^2.
    \end{align}
    Now form \eqref{Eq:2ndGrad1}, using the tower rule yields
    \begin{align}
    \begin{split}
    &\mathbb{E} \left[ \| \mathbf{g}(\mathbf{w}_k,\bm{\xi}_k) + \frac{1}{\gamma_k} \left(\mathcal{L}\otimes I_{d_w}\right)\mathbf{w}_k \|_2^2 \right] \leq
    \mu_{v_1} \\&+ \mu_{v_2} \mathbb{E} \left[\left\|  \nabla F(\mathbf{w}_k) + \frac{1}{\gamma_k} \left(\mathcal{L}\otimes I_{d_w}\right)\mathbf{w}_k \right\|^2_2 \right]
    \end{split}\label{Eq:2ndGrad2}
    \end{align}
    Now taking the expectation of \eqref{Eq:DSG3} and substituting \eqref{Eq:2ndGrad2} yields
    \begin{align}\label{InEq:DSG1}
      \mathbb{E} \left[ \left\| \mathbf{w}_{k + 1} - \mathbf{w}_{k} \right\|_2^2 \right] &\leq \alpha_k^2 \mu_{v_1} + \alpha_k^2 \mu_{v_2} \mathbb{E} \left[\left\|  \nabla V(\mathbf{w}_k) \right\|^2_2 \right].
    \end{align}
    Thus we have
    \begin{align}\label{InEq:DSG2}
    \begin{split}
      &\sum\limits_{k=0}^{\infty} \, \mathbb{E} \left[ \left\| \mathbf{w}_{k + 1} - \mathbf{w}_{k} \right\|_2^2 \right] \leq \sum\limits_{k=0}^{\infty} \,\left(\alpha_k^2 \mu_{v_1}\right) \\
      &\qquad \qquad + \sum\limits_{k=0}^{\infty} \, \left(\alpha_k^2 \mu_{v_2} \mathbb{E} \left[\left\|  \nabla V(\mathbf{w}_k) \right\|^2_2 \right]\right).
    \end{split}
    \end{align}
    Now \eqref{InEq:DSG3} follows from \eqref{Eqn:SummableGrad} and from noting that $\alpha_k$ is square summable. Furthermore, since every summable sequence is convergent, we have \eqref{ConvergenceEqn}.

\subsection{Proof of Theorem~\ref{Theorem:SummableGrad}}\label{Appdx:E}
Taking the conditional expectation $\mathbb{E}_{\xi}[\cdot]$ of \eqref{Eq:DSG2}  yields
%\begin{align}\label{Eq:DSG2a}
%  \mathbb{E}_{\xi}\left[\, \mathbf{w}_{k + 1} \,\right] - \mathbf{w}_{k} &= - \alpha_k \left( \mathbb{E}_{\xi}\left[\, \mathbf{g}(\mathbf{w}_k,\bm{\xi}_k) \,\right] + \frac{1}{\gamma_k} \left(\mathcal{L}\otimes I_{d_w}\right)\mathbf{w}_k \right)\\
%  &= - \alpha_k \left( \nabla F(\,\mathbf{w}_k\,) + \frac{1}{\gamma_k} \left(\mathcal{L}\otimes I_{d_w}\right)\mathbf{w}_k \right)\quad \textnormal{a.s.}\\
%  &= - \alpha_k \nabla V(\gamma_k,\mathbf{w}_k) \quad \textnormal{a.s.}
%\end{align}
%Thus we have
\begin{align}\label{Eq:DSG2b}
  \mathbb{E}_{\xi}\left[\, \mathbf{w}_{k + 1}  - \mathbf{w}_{k} \,\right]
  &=  - \alpha_k \nabla V(\gamma_k,\mathbf{w}_k) \quad \textnormal{a.s.}
\end{align}
Thus we have
\begin{align}\label{Eq:DSG23}
  \left\| \mathbb{E}_{\xi}\left[\, \mathbf{w}_{k + 1}  - \mathbf{w}_{k} \,\right] \right\|^2_2
  &=  \alpha_k^2  \left\| \nabla V(\gamma_k,\mathbf{w}_k) \right\|^2_2 \quad \textnormal{a.s.}
\end{align}
Therefore
\begin{align}\label{Eq:DSG2d}
  \mathbb{E}\left[ \left\|  \nabla V(\gamma_k,\mathbf{w}_k)  \right\|^2_2 \right] = \alpha_k^{-2} \, \mathbb{E}\left[ \left\| \mathbb{E}_{\xi}\left[\, \mathbf{w}_{k + 1}  - \mathbf{w}_{k} \,\right] \right\|^2_2 \right]
\end{align}
Substituting \eqref{Eq:DSG2d} into \eqref{Eqn:SummableGrad} yields
\begin{equation}\label{Eqn:Summabledw}
    \sum\limits_{k=0}^{\infty} \, \alpha_k^{-1}\, \mathbb{E}\left[ \left\| \mathbb{E}_{\xi}\left[\, \mathbf{w}_{k + 1}  - \mathbf{w}_{k} \,\right] \right\|^2_2 \right] < \infty.
\end{equation}
Now note that $\bar{\mathbf{w}}_{k + 1} - \bar{\mathbf{w}}_{k}  =  \frac{1}{n} \left(\mathbf{1}_n \mathbf{1}_n^\top \otimes I_{d_w} \right) \left(\mathbf{w}_{k + 1} - \mathbf{w}_{k}\right)$. Thus
$\mathbb{E}_{\xi}\left[\ \bar{\mathbf{w}}_{k + 1} - \bar{\mathbf{w}}_{k} \right]  =  \frac{1}{n} \left(\mathbf{1}_n \mathbf{1}_n^\top \otimes I_{d_w} \right) \mathbb{E}_{\xi}\left[\ \mathbf{w}_{k + 1} - \mathbf{w}_{k}\right]$ a.s. and $\left\| \mathbb{E}_{\xi}\left[\ \bar{\mathbf{w}}_{k + 1} - \bar{\mathbf{w}}_{k} \right] \right\|_2 \leq  \left\| \mathbb{E}_{\xi}\left[\ \mathbf{w}_{k + 1} - \mathbf{w}_{k}\right] \right\|_2$ a.s.
%\begin{align*}
%   \left\| \mathbb{E}_{\xi}\left[\ \bar{\mathbf{w}}_{k + 1} - \bar{\mathbf{w}}_{k} \right] \right\|_2
%   %&=  \left\| \frac{1}{n} \left(\mathbf{1}_n \mathbf{1}_n^\top \otimes I_{d_w} \right) \mathbb{E}_{\xi}\left[\ \mathbf{w}_{k + 1} - \mathbf{w}_{k}\right] \right\|_2 \quad \textnormal{a.s.}\\
%   &\leq  \left\| \mathbb{E}_{\xi}\left[\ \mathbf{w}_{k + 1} - \mathbf{w}_{k}\right] \right\|_2 \quad \textnormal{a.s.}
%\end{align*}
Therefore it follows from \eqref{Eqn:Summabledw} that
\begin{equation}\label{Eqn:Summabledwbar}
    \sum\limits_{k=0}^{\infty} \, \alpha_k^{-1}\, \mathbb{E}\left[ \left\| \mathbb{E}_{\xi}\left[\, \bar{\mathbf{w}}_{k + 1}  - \bar{\mathbf{w}}_{k} \,\right] \right\|^2_2 \right] < \infty.
\end{equation}
From \eqref{Eq:DSG2} we have
\begin{align}\label{Eq:DSGbar}
  \mathbb{E}_{\xi}\left[\, \bar{\mathbf{w}}_{k + 1}  - \bar{\mathbf{w}}_{k} \,\right]
  &= - \alpha_k \overline{\nabla F} (\mathbf{w}_k) \quad \textnormal{a.s.}
\end{align}
Now substituting \eqref{Eq:DSGbar} into \eqref{Eqn:Summabledwbar} yields \eqref{Eqn:SummableDFbar}.

\subsection{Proof of Theorem~\ref{Theorem:OptCond}}\label{Appdx:F}
Define $G(\mathbf{w}_k) \triangleq \left\| \overline{\nabla F} (\mathbf{w}_k) \right\|^2_2$. Thus we have
%\begin{align}
%    G(\mathbf{w}_k) = \overline{\nabla F} (\mathbf{w}_k)^\top \overline{\nabla F} (\mathbf{w}_k) = \nabla F(\mathbf{w}_k)^\top \mathcal{J}\, \nabla F(\mathbf{w}_k),
%\end{align}
%and
\begin{align}
    \nabla G(\mathbf{w}_k) &= 2\nabla^2 F(\mathbf{w}_k) \mathcal{J}\, \nabla F(\mathbf{w}_k),\label{GradG}
\end{align}
where $\mathcal{J}\, = \left( \frac{1}{n} \left(\mathbf{1}_n \mathbf{1}_n^\top \otimes I_{d_w} \right) \right)$ and $\mathcal{J}^2 = \mathcal{J}\,$. Since $F(\cdot)$ is twice continuously differentiable and $\nabla F(\cdot)$ is Liptschitz continuous with constant $L$, we have $ \nabla^2 F(\mathbf{w}) \leq L I_{nd_w}$. Therefore $\forall \,\mathbf{w}_a,\,\mathbf{w}_b\in\mathbb{R}^{n d_w}$,
\begin{align*}
    &\nabla G(\mathbf{w}_{a}) - \nabla G(\mathbf{w}_b) = 2\nabla^2 F(\mathbf{w}_{a}) \mathcal{J}\, \nabla F(\mathbf{w}_{a}) - 2\nabla^2 F(\mathbf{w}_b) \mathcal{J}\, \\
    &\times \nabla F(\mathbf{w}_b) +     2\nabla^2 F(\mathbf{w}_{a}) \mathcal{J}\, \nabla F(\mathbf{w}_{b}) - 2\nabla^2 F(\mathbf{w}_{a}) \mathcal{J}\, \nabla F(\mathbf{w}_{b}) \\
    & = 2\nabla^2 F(\mathbf{w}_{a}) \mathcal{J}\, \left(\nabla F(\mathbf{w}_{a}) - \nabla F(\mathbf{w}_{b}) \right) \\
    & \qquad\qquad\qquad + 2\left( \nabla^2 F(\mathbf{w}_{a}) - \nabla^2 F(\mathbf{w}_{b}) \right) \mathcal{J}\, \nabla F(\mathbf{w}_{b})
\end{align*}
Since $\nabla^2 F(\mathbf{w}_{a})$ is Lipschitz continuous with constant $L_H$, and $\nabla F(\mathbf{w}_{b}) \leq \mu_F$, we have
\begin{align*}
    &\left\| \nabla G(\mathbf{w}_{a}) - \nabla G(\mathbf{w}_b) \right\|_2 \leq 2L^2 \left\| \mathbf{w}_{a} - \mathbf{w}_{b} \right\|_2 \\
    &\qquad\qquad\qquad  + 2\mu_F L_H \left\| \mathbf{w}_{a} - \mathbf{w}_{b} \right\|_2\leq L_G \left\| \mathbf{w}_{a} - \mathbf{w}_{b} \right\|_2,
\end{align*}
where $L_G \geq 2L^2 + 2\mu_F L_H$. Thus $\nabla G(\mathbf{w})$ is Lipschitz continuous and from Lemma~\ref{Lemma:Lipz} we have
\begin{align*}
    G(\mathbf{w}_{k+1}) \leq  G(\mathbf{w}_k) &+ \nabla G(\mathbf{w}_k)^\top \left( \mathbf{w}_{k+1} - \mathbf{w}_{k} \right)\\ &+ \frac{1}{2}L_G \left\| \mathbf{w}_{k+1} - \mathbf{w}_{k} \right\|_2^2
\end{align*}
Now substituting \eqref{GradG} and taking the conditional expectation $\mathbb{E}_{\xi}[\,\cdot\,]$ yields
\begin{align*}
    \mathbb{E}_{\xi}\left[\,G(\mathbf{w}_{k+1}) \,\right] \leq  &G(\mathbf{w}_k) + \frac{1}{2}L_G \mathbb{E}_{\xi}\left[\,\left\| \mathbf{w}_{k+1} - \mathbf{w}_{k} \right\|_2^2 \,\right] \\
    &\quad + 2\nabla F(\mathbf{w}_k)^\top \mathcal{J}\, \nabla^2 F(\mathbf{w}_k) \mathbb{E}_{\xi}\left[\, \mathbf{w}_{k+1} - \mathbf{w}_{k} \,\right]
\end{align*}
Since $\nabla F(\mathbf{w}_k)^\top \mathcal{J}\, = \nabla V(\gamma_k,\mathbf{w}_k)^\top \mathcal{J}$, substituting \eqref{Eq:DSG2b} yields
\begin{align*}
    &\mathbb{E}_{\xi}\left[\,G(\mathbf{w}_{k+1}) \,\right] \leq  G(\mathbf{w}_k)  + \frac{1}{2}L_G \mathbb{E}_{\xi}\left[\,\left\| \mathbf{w}_{k+1} - \mathbf{w}_{k} \right\|_2^2 \,\right] \\
    &\qquad \qquad \qquad - 2\alpha_k \,\nabla V(\gamma_k,\mathbf{w}_k)^\top \mathcal{J}\, \nabla^2 F(\mathbf{w}) \nabla V(\gamma_k,\mathbf{w}_k)\\
    &\leq G(\mathbf{w}_k) + 2\alpha_k L \left\| \nabla V(\gamma_k,\mathbf{w}_k) \right\|^2_2 + \frac{1}{2}L_G \mathbb{E}_{\xi}\left[\,\left\| \mathbf{w}_{k+1} - \mathbf{w}_{k} \right\|_2^2 \,\right]
\end{align*}
Now taking the total expectation yields
\begin{align}\label{Eqn:DiffG}
\begin{split}
    \mathbb{E}\left[\,G(\mathbf{w}_{k+1}) \,\right]
    \leq \mathbb{E}\left[\, G(\mathbf{w}_k) \,\right] &+ 2\alpha_k L \mathbb{E}\left[\,\left\| \nabla V(\gamma_k,\mathbf{w}_k) \right\|^2_2 \,\right] \\&+ \frac{1}{2}L_G \mathbb{E}\left[\,\left\| \mathbf{w}_{k+1} - \mathbf{w}_{k} \right\|_2^2 \,\right]
\end{split}
\end{align}
From \eqref{Eqn:SummableGrad} and \eqref{InEq:DSG3}, we know that $\alpha_k \mathbb{E}\left[\,\left\| \nabla V(\gamma_k,\mathbf{w}_k) \right\|^2_2 \,\right]$ and $\mathbb{E}\left[\,\left\| \mathbf{w}_{k+1} - \mathbf{w}_{k} \right\|_2^2 \,\right]$ are summable. Therefore \eqref{Eqn:DiffG} can be written in the form of \eqref{Eq:Robbins} and it follows from Lemma~\ref{Lemma:Robbins} that $\mathbb{E}\left[\, G(\mathbf{w}_k) \,\right]$ converges. Since $ \mathbb{E}\left[\, G(\mathbf{w}_k) \,\right] = \mathbb{E}\left[\, \left\| \overline{\nabla F} (\mathbf{w}_k) \right\|^2_2 \,\right]$ it follows from Theorem~\ref{Theorem:SummableGrad} that $\mathbb{E}\left[\, G(\mathbf{w}_k) \,\right]$ must converge to zero.

\bibliography{Biblio}
\bibliographystyle{IEEEtran}

\end{document}